\documentclass[12pt]{amsart}
\usepackage{a4wide}
\usepackage[T1]{fontenc}
\usepackage{amssymb,amsmath,amsthm,latexsym}
\usepackage{mathrsfs}
\usepackage[usenames,dvipsnames]{xcolor}
\usepackage{graphicx}
\usepackage{mdwlist}
\usepackage{enumerate}
\usepackage{mathtools,dsfont,wasysym}
\usepackage{enumitem}
\usepackage[all]{xy}
\usepackage{multicol}
\usepackage{tikz}
\usepackage[pagebackref]{hyperref}

\hypersetup{colorlinks=true,  linkcolor=blue, citecolor=blue, urlcolor=cyan}

\newtheorem{theorem}{Theorem}[section]
\newtheorem{proposition}[theorem]{Proposition}

\newtheorem{problem}[theorem]{Problem}
\newtheorem{claim}[theorem]{Claim}
\newtheorem{lemma}[theorem]{Lemma}

\theoremstyle{remark}
\newtheorem{remark}[theorem]{Remark}
\theoremstyle{definition}
\newtheorem{definition}[theorem]{Definition}
\newtheorem{example}[theorem]{Example}
\numberwithin{equation}{section}
\makeatother

\newcommand{\R}{\mathbb{R}}
\newcommand{\N}{\mathbb{N}}
\newcommand{\e}{\varepsilon}
\newcommand{\p}{\varphi}
\newcommand{\n}{\left\Vert\cdot\right\Vert}
\newcommand{\nn}[1]{{\left\vert\kern-0.25ex\left\vert\kern-0.25ex\left\vert #1 
\right\vert\kern-0.25ex\right\vert\kern-0.25ex\right\vert}}
\newcommand{\bone}{\text{\usefont{U}{bbold}{m}{n}1}}
\newcommand{\cut}{\mathord{\upharpoonright}}
\renewcommand{\leq}{\leqslant}
\renewcommand{\geq}{\geqslant}

\newcommand{\ro}{\varrho}
\newcommand{\X}{\mathcal{X}}
\newcommand{\Z}{\mathcal{Z}}
\newcommand{\C}{\mathcal{C}}
\newcommand{\K}{\mathcal{K}}
\newcommand{\F}{\mathcal{F}}
\newcommand{\ev}{\mathsf{ev}}
\renewcommand\qedsymbol{$\blacksquare$} 
\newcounter{smallromans}

\newenvironment{romanenumerate}
{\begin{list}{{\normalfont\textrm{(\roman{smallromans})}}}%
  {\usecounter{smallromans}\setlength{\itemindent}{0cm}%
   \setlength{\leftmargin}{5.5ex}\setlength{\labelwidth}{5.5ex}%
   \setlength{\topsep}{.5ex}\setlength{\partopsep}{.5ex}%
   \setlength{\itemsep}{0.1ex}}}%
{\end{list}}

\usepackage{xpatch}
\makeatletter   
\xpatchcmd{\@tocline}
{\hfil\hbox to\@pnumwidth{\@tocpagenum{#7}}\par}
{\ifnum#1<0\hfill\else\dotfill\fi\hbox to\@pnumwidth{\@tocpagenum{#7}}\par}{}{}
\makeatother

\begin{document}
\title[Norming M-bases: recent results and open problems]{Norming Markushevich bases:\\recent results and open problems}

\author[P.~H\'ajek]{Petr H\'ajek}
\address[P.~H\'ajek]{Department of Mathematics, Faculty of Electrical Engineering, Czech Technical University in Prague, Technick\'a 2, 166 27 Prague 6, Czech Republic}
\email{hajekpe8@fel.cvut.cz}

\author[T.~Russo]{Tommaso Russo}
\address[T.~Russo]{Universit\"{a}t Innsbruck, Department of Mathematics, Technikerstra\ss e 13, 6020 Innsbruck, Austria \newline \href{https://orcid.org/0000-0003-3940-2771}{ORCID: \texttt{0000-0003-3940-2771}}}
\email{tommaso.russo@uibk.ac.at, tommaso.russo.math@gmail.com}

\dedicatory{Dedicated to Gilles Godefroy on the occasion of his 70th birthday}
\thanks{Research of P.~H\'ajek was supported in part by GA23-04776S and by grant SGS22/053/OHK3/1T/13 of CTU in Prague.}

\date{\today}

\begin{abstract} We survey several results concerning norming Markushevich bases (M-bases, for short), focusing in particular on two recent examples of a weakly compactly generated Banach space with no norming M-basis and of an Asplund space with norming M-basis that is not weakly compactly generated. We highlight the context for these problems and state several open problems in different directions that arise from these results.
\end{abstract}
\maketitle
\tableofcontents

%-------------------------------------------------------%
%                                                       %
% 						INTRODUCTION 					%
%                                                       %
%-------------------------------------------------------%
\section{Introduction}
A sequence $(e_k)_{k=1}^\infty$ in a Banach space $\X$ is a \emph{Schauder basis} if every $x\in \X$ admits a unique representation
\[ x=\sum_{k=1}^\infty x_k e_k, \]
where the series is assumed to converge in the norm of $\X$. This permits an identification between the vector $x\in \X$ and the scalar sequence $(x_k)_{k=1}^\infty$ and allows to view the Banach space $\X$ as a more concrete object whose elements are sequences of reals. Among the advantages of this representation, let us mention for instance that algebraic manipulations become more handy and linear operators can be represented by infinite matrices. Accordingly, Schauder bases have been extensively exploited in Banach space theory to obtain several structural results, see, \emph{e.g.}, \cite{AK, FHHMZ, Singer1}. Despite their usefulness, they have at least two drawbacks: first, the existence of a Schauder basis in $\X$ clearly implies that $\X$ must be separable; second, even though every classical separable Banach space admits one, there are separable Banach spaces without Schauder bases, \cite{Enflo}.

A well-known, yet non-trivial, result due to Banach is that the coefficient functionals $\p_k\colon \X\to \R$ given by $x\mapsto x_k$ are continuous. Additionally, ${\rm span}\{e_k\}_{k=1}^\infty$ is dense in $\X$ and the functionals $\{\p_k\}_{k=1}^\infty$ separate points on $\X$. Therefore, one can consider the following generalisation of a Schauder basis. A sequence $\{e_k;\p_k\}_{k=1}^\infty\subseteq \X\times \X^*$ is a \emph{Markushevich basis} (hereinafter \emph{M-basis}) if $\langle\p_k, e_n\rangle= \delta_{k,n}$, ${\rm span}\{e_k\}_{k=1}^\infty$ is dense in $\X$, and ${\rm span}\{\p_k\}_{k=1}^\infty$ is $w^*$-dense in $\X^*$. This more general definition solves at once the two drawbacks of Schauder bases from the previous paragraph: indeed, it is clear how to formulate the definition of an M-basis in a way that would not require separability, and M-bases are known to exist in all separable Banach spaces \cite{Markushevich} (we will say more on this in Section \ref{sec: example}). In addition, there are natural examples of M-bases that are not Schauder bases, such as, for example, the classical trigonometric system from Fourier Analysis, see Example \ref{ex: trig system}. More importantly, M-bases have been successfully exploited to prove results valid for all separable Banach spaces; as a small sample, let us mention \cite{Ansari, DBS Rotund not LUR, DB FPP, HR densely, HaTa, Herzog}. Before we continue, we should mention the main weakness of M-bases compared to Schauder bases: while it remains true that every vector $x\in \X$ can be unambiguously identified with its coordinates $(\langle\p_k, x\rangle) _{k=1}^\infty$, the series $\sum_{k=1}^\infty \langle\p_k, x\rangle e_k$ need not converge in general. \smallskip

However, it is the non-separable framework the one where M-bases have proved being the most fruitful. As a first instance of this, let us mention that every Banach space with an M-basis admits a continuous injection in $c_0(\Gamma)$ and therefore a rotund renorming, see, \emph{e.g.}, \cite[Theorem II.2.4]{DGZ}. In order to achieve stronger renorming properties one frequently requires additional assumptions on the M-basis. Perhaps the most important example is the celebrated Troyanski's renorming theorem \cite{Troyanski LUR} that every Banach space with a strong M-basis admits a locally uniformly rotund (LUR, for short) norm. Let us also refer to \cite[Theorem 3.48]{HMVZ} for a proof and to \cite{DHR fundamental}, \cite{MOTV}, and \cite[Lemma 2.1]{Moreno fund LUR} for related results. Among all properties of M-bases the closest one to renorming theory is that of a norming M-basis, since the very definition of norming M-basis involves an equivalent norm induced by the M-basis itself (see Section \ref{sec: example} for details).

A related use of M-bases in non-separable Banach spaces is to characterise several classes of non-separable Banach spaces by the existence of M-bases with certain additional properties. The most important examples of this are weakly Lindel\"of determined (WLD, for short) and Plichko spaces that we will define in Section \ref{sec: non-sep classes} and will play a role in several parts of the paper. Let us refer to \cite[Chapter 6]{HMVZ} and \cite{Zizler} for more on this perspective; let us also mention in passing a parallel line of research aiming at characterising classes of Banach spaces via certain properties of so called projectional skeletons, see, \emph{e.g.}, \cite{CCS projection, CF skeleton, FM skeleton}.

One of the main open problems in the area, still widely open, has been to detect which class of Banach spaces is characterised by the existence of a norming M-basis. It has been conjectured in the Seventies that there ought to be a connection between norming M-bases and weakly compactly generated Banach spaces (WCG for short). This resulted in a famous open problem, due to John and Zizler \cite{JZ norming WCG}, whether every WCG Banach space admits a norming M-basis and to a question in the opposite direction, due to Godefroy \cite[Problem~112]{GMZ open}, whether Asplund spaces with norming M-bases are WCG. Both these problems have been recently solved in the negative \cite{Hajek, HRST} and we shall give more details on this in Section \ref{sec: WCG}. On the other hand, we also prove in Section \ref{sec: WCG} that the answer to Godefroy's problem is positive under the stronger assumption that $\X$ has countable Szlenk index (Theorem \ref{thm: countable Sz}).

As it turns out, the example from \cite{Hajek} of a WCG Banach space without norming M-basis is even a $\C(\K)$ space, while the construction from \cite{HRST} is based on the construction of a certain compact space, denoted by $\K_\ro$, and then the desired space is a subspace of $\C(\K_\ro)$. This left open the natural problem whether it is possible to obtain an actual $\C(\K)$ space as a counterexample, namely if there is an Asplund $\C(\K)$ space with norming M-basis that is not WCG. In Section \ref{sec: C(K)} we shall discuss some partial results in this direction, taken from \cite{RS C(K)}, and state several open questions that originate from a natural approach to the problem.

In turn, this question is connected to the heredity problem for norming M-bases, \emph{i.e.}, the problem whether the existence of norming M-bases passes to subspaces, that we shall focus on in Section \ref{sec: subspace}. To the best of our knowledge, this problem is still open and we will mention some natural candidates for a counterexample and some classes where a positive result is available. For now, let us just mention two results concerning similar questions: it is known that the existence of an M-basis does not pass to subspaces, while shrinking M-bases do pass to subspaces. \smallskip

The last section of the paper, Section \ref{sec: semi-Eberlein}, focuses on the connection between M-bases, particularly norming M-bases, and descriptive topology. To explain this connection, let us begin with a simple standard fact. Assume that $\{e_k;\p_k\}_{k=1}^\infty$ is an M-basis of a separable Banach space $\X$ and that $\|e_k\|=1$ for all $k\in \N$. Then the evaluation map
\begin{equation}\label{eq: ev map}
    \ev\colon (B_{\X^*},w^*) \to [-1,1]^\N\colon \qquad \p \mapsto (\langle\p, e_k\rangle)_{k=1}^\infty
\end{equation}
is easily seen to be a continuous injection (here, $[-1,1]^\N$ is endowed with the product topology). By the compactness of $(B_{\X^*},w^*)$, it is a homeomorphic embedding, which implies the standard fact that $(B_{\X^*},w^*)$ is metrisable. The same evaluation map $\ev$, now having values in $[-1,1]^\Gamma$, can be defined in presence of any M-basis and it still is a homeomorphic embedding. Depending on the range of $\ev$ one can classify several classes of Banach spaces: for example, when $\X$ is WLD $\ev(\p)$ only has countably many non-zero coordinates, which leads to the notion of a Corson compact space. Similar results are available when $\X$ is WCG, Hilbert generated, or Plichko, as we will explain in Section \ref{sec: non-sep classes}. When the M-basis is $1$-norming, the description of the image of $\ev$ leads to the notion of \emph{semi-Eberlein} compact space, first introduced in \cite{KL}. In Section \ref{sec: semi-Eberlein} we shall define this class of compact spaces, describe some known properties, and state several open problems.

%-------------------------------------------------------%
%                                                       %
% 						  EXAMPLES   					%
%                                                       %
%-------------------------------------------------------%
\section{Examples of M-bases}\label{sec: example}
In this section we recall the definitions of various notions related to norming M-bases and give some natural examples of them; in particular, we briefly discuss the existence of (norming) M-bases in separable Banach spaces. \smallskip

For the sake of completeness, we begin by explicitly stating the definition of an M-basis without the separability constraint. Given a Banach space $\X$, a family $\{e_\alpha;\p_\alpha\}_{\alpha\in \Gamma}\subseteq \X\times \X^*$ is a \emph{biorthogonal system} in $\X$ if $\langle\p_\alpha, e_\beta\rangle= \delta_{\alpha,\beta}$ for all $\alpha,\beta \in\Gamma$. The biorthogonal system $\{e_\alpha;\p_\alpha\}_{\alpha\in \Gamma}$ is \emph{fundamental} (or \emph{complete}) if ${\rm span} \{e_\alpha\}_{\alpha\in \Gamma}$ is dense in  $\X$ and it is \emph{total} if ${\rm span} \{\p_\alpha\}_{\alpha\in \Gamma}$ is $w^*$-dense in  $\X^*$ (equivalently, if it separates points of $\X$: the unique vector $x\in \X$ with $\langle\p_\alpha, x\rangle=0$ for all $\alpha\in \Gamma$ is $x=0$). A \emph{Markushevich basis} (henceforth, \emph{M-basis}) is a fundamental and total biorthogonal system. As we said in the Introduction, one frequently requires M-bases with additional properties, that are frequently obtained by strengthening the sense in which the functionals $\{\p_\alpha\} _{\alpha\in \Gamma}$ exhaust the dual. The strongest notion is that of \emph{shrinking} M-basis, obtained by requiring that ${\rm span} \{\p_\alpha\}_{\alpha\in \Gamma}$ is dense in  $\X^*$ (in the norm topology); as we will see in Theorem \ref{thm: shrinking M-basis} this is a rather restrictive definition.

We now introduce the main definition for the paper. Recall that a subspace $\Z$ of $\X^*$ is \emph{norming} if the formula
\[ \|x\|_\Z\coloneqq \sup\{ |\langle\p,x\rangle| \colon \p\in \Z, \|\p\|\leq 1\} \]
defines an equivalent norm on $\X$ (notice that we always have $\|x\|_\Z\leq \|x\|$). Quantitatively, $\Z$ is \emph{$\lambda$-norming} (for $0<\lambda\leq 1$) if
\[ \lambda\|x\|\leq \sup\{ |\langle\p,x\rangle| \colon \p\in \Z, \|\p\|\leq 1\}. \]
By the Hahn--Banach theorem, this is easily seen to be equivalent to
\begin{equation}\label{eq: norming in dual ball}
    \lambda B_{\X^*}\subseteq \overline{\Z\cap B_{\X^*}}^{w^*}.
\end{equation}
Observe that $\Z$ becomes a $1$-norming subspace in the norm $\n_\Z$. Let us refer to \cite[Section 3.4]{GMZ renorm} for some additional elementary remarks on the notion of norming subspaces.

\begin{definition} An M-basis $\{e_\alpha; \p_\alpha\}_{\alpha\in\Gamma}$ in $\X$ is \emph{$\lambda$-norming} ($0<\lambda \leq1$) if ${\rm span}\{\p_\alpha\}_{\alpha\in\Gamma}$ is a $\lambda$-norming subspace of $\X^*$, namely if
\begin{equation}\label{eq: norming M-basis}
    \lambda\|x\|\leq \sup\big\{|\langle\p,x\rangle|\colon \p\in {\rm span}\{\p_\alpha\}_{\alpha\in\Gamma},\, \|\p\|\leq 1 \big\} \qquad(x\in\X).
\end{equation}
\end{definition}

Note that in order to prove that an M-basis is $\lambda$-norming, it is enough to verify \eqref{eq: norming M-basis} for a dense subset of $\X$. Moreover, since every norming subspace is $w^*$-dense, the norming property automatically implies that the biorthogonal system is total. Let us now turn to some natural examples of norming M-bases.
\begin{example}\label{ex: trig system} Let $\mathbb{T}$ be the unit circle in the complex plane and $\mu$ be the normalised Lebesgue measure on $\mathbb{T}$. We shall prove that the trigonometric system is a $1$-norming M-basis in $L_1(\mathbb{T},\mu)$ and $\C(\mathbb{T})$. Consider the functions $e_k(t)\coloneqq e^{ikt}$, $k\in\mathbb{Z}$, that span a dense subspace of both $L_1(\mathbb{T},\mu)$ and $\C(\mathbb{T})$. The biorthogonal functionals $\p_k$ are given by $\p_k=e_{-k}$; more precisely, in $\C(\mathbb{T})$ we should consider the measure $e_{-k}\ {\rm d} \mu$. Thus, $\{e_k;\p_k\}_{k\in\mathbb{Z}}$ is a fundamental biorthogonal system for both $L_1(\mathbb{T},\mu)$ and $\C(\mathbb{T})$; note that ${\rm span} \{\p_k\}_{k\in\mathbb{Z}}$ is the set of trigonometric polynomials.

We now check the norming property. Let $(F_n)_{n=1}^\infty$ be the Fej\'er kernel (see, \emph{e.g.}, \cite{Grafakos} for the definitions and the properties that we shall use). If $f\in L_1(\mathbb{T},\mu)$, take $g\in L_\infty(\mathbb{T},\mu)$ such that $\|g\|_\infty= 1$ and $fg= |f|$. Then the convolution $F_n*g$ converges $\mu$-a.e. to $g$ and $\|F_n*g\|_\infty\leq 1$. Thus by the dominated converge theorem
\[ \int_\mathbb{T} \big(F_n*g \big)f\ {\rm d}\mu \to \|f\|_1 \]
and the trigonometric polynomials $F_n*g$ prove that ${\rm span} \{\p_k\}_{k\in\mathbb{Z}}$ is $1$-norming.

If $f\in \C(\mathbb{T})$, take $\alpha,t_0\in \mathbb{T}$ such that $\|f\|_\infty= |f(t_0)|= \alpha f(t_0)$; then
\[ \int_\mathbb{T} \alpha F_n(t_0-t)f(t)\ {\rm d}\mu(t)= \alpha (F_n*f)(t_0)\to \|f\|_\infty. \]
Thus the trigonometric polynomials $t\mapsto \alpha F_n(t_0-t)$ (that are normalised in $L_1(\mathbb{T},\mu)$) imply that the M-basis is $1$-norming in this case as well.

More generally, $\C(\K)$ has a $1$-norming M-basis whenever $\K$ is a compact Abelian group, \cite[Theorem 4.6]{Kalenda natural}.
\end{example}

\begin{example} Every Schauder basis is a norming M-basis and every monotone Schauder basis is a $1$-norming one. This is elementary and we refer to \cite[Remark 32(3)]{GMZ renorm} for a proof. A similar argument shows that every long unconditional basis is a norming M-basis as well, \cite[Exercise 7.8]{HMVZ}. In particular, the canonical bases of $c_0(\Gamma)$ and $\ell_p(\Gamma)$, $1\leq p<\infty$ are $1$-norming. In the non-separable case the unconditionality assumption is essential, as not every long Schauder basis is norming (see Remark \ref{rmk: long Schauder not norming}).
\end{example}

In the opposite direction, the classical Mazur technique shows that every norming M-basis admits a Schauder subsequence; we state this only for countable M-bases, but a small variation gives the result for uncountable ones as well, \cite[Fact 4.4]{HRST}.
\begin{lemma} Let $\{e_k;\p_k\}_{k=1}^\infty$ be a norming M-basis for $\X$. Then some subsequence of $(e_k)_{k=1}^\infty$ is a basic sequence in $\X$. 
\end{lemma}
\begin{proof} Up to renorming $\X$, we assume that $\{e_k;\p_k\}_{k=1}^\infty$ is $1$-norming. If $\Z$ is a finite-dimensional subspace of $\X$ and $\e>0$, by definition there exists $n\in\N$ with 
\[ (1-\e)\|x\|\leq \sup\big\{|\langle\p,x\rangle|\colon \p\in {\rm span}\{\p_k\}_{k=1}^n,\, \|\p\|\leq 1 \big\} \qquad(x\in\Z). \]
Therefore, for $x\in \Z$ and $t\in \R$ we have
\[ (1-\e)\|x\|\leq \sup\big\{|\langle\p,x + te_{n+1}\rangle|\colon \p\in {\rm span}\{\p_k\}_{k=1}^n,\, \|\p\|\leq 1 \big\}\leq \|x + te_{n+1}\| \]
and a standard inductive argument (as in \cite[Theorem 4.19]{FHHMZ}) finishes the proof.
\end{proof}

Next, we discuss the existence of (norming) M-bases in separable Banach spaces. The following crucial result, essentially based on the Gram-Schmidt orthogonalisation algorithm, is due to Markushevich \cite{Markushevich}. We refer to \cite[Theorem 4.59]{FHHMZ} for a proof.
\begin{theorem}[\cite{Markushevich}]\label{thm: separable M-basis} Let $\X$ be a separable Banach space and let $\{x_k\}_{k=1}^\infty\subseteq \X$ and $\{\psi_k\}_{k=1}^\infty\subseteq \X^*$ be such that ${\rm span}\{x_k\}_{k=1}^\infty$ is dense in $\X$ and ${\rm span}\{\psi_k\}_{k=1}^\infty$ is $w^*$-dense in $\X^*$. Then there exists an M-basis $\{e_k;\p_k\}_{k=1}^\infty$ of $\X$ such that
\[ {\rm span}\{e_k\}_{k=1}^\infty= {\rm span}\{x_k\}_{k=1}^\infty \qquad\mbox{and}\qquad {\rm span}\{\p_k\}_{k=1}^\infty= {\rm span}\{\psi_k\}_{k=1}^\infty. \]
\end{theorem}

The result immediately implies the existence of a $1$-norming M-basis in every separable Banach space. Indeed, it is enough to take a dense sequence $(x_k)_{k=1}^\infty$ in $\X$, let $\psi_k$ be a norm one functional with $\langle\psi_k, x_k\rangle=1$ and observe that ${\rm span}\{\psi_k\}_{k=1}^\infty$ is $1$-norming for $\X$.

The situation is radically different for non-separable $\X$ and there are non-separable Banach spaces without an M-basis, the simplest example of this being $\ell_\infty$, \cite{Johnson}. It is even consistent with ZFC that there are non-separable Banach spaces without uncountable biorthogonal systems, \cite[Section 4.4]{HMVZ}.

Part of the interest of Theorem \ref{thm: separable M-basis} is in the clause concerning the preservation of the linear spans, that allows for instance to obtain examples of non-norming M-bases.

\begin{remark} When searching for M-bases that fail to be norming one immediately faces the problem of finding $w^*$-dense subspaces that are not norming. The existence of such subspaces has been clarified in \cite{DL total not norming}: A Banach space $\X$ has a $w^*$-dense not norming subspace in $\X^*$ if and only if $\X^{**}/ \X$ is infinite-dimensional (\emph{i.e.}, $\X$ is not quasi-reflexive). Therefore, by Theorem \ref{thm: separable M-basis} every separable Banach space that is not quasi-reflexive admits a non-norming M-basis. Conversely, in a separable quasi-reflexive Banach space all M-bases are norming (and in a separable reflexive space all M-bases are shrinking).
\end{remark}

A similar argument shows the existence of norming M-bases that are not strong. An M-basis $\{e_\alpha;\p_\alpha\}_{\alpha\in \Gamma}$ is \emph{strong} if for every $x\in \X$ one has
\[ x\in \overline{\rm span}\{e_\alpha\colon \langle\p_\alpha, x\rangle \neq0\}. \]
For instance, the trigonometric system of Example \ref{ex: trig system} is strong. Despite how innocent this condition might look there are M-bases that fail to be strong and it is actually rather easy to find one in every separable Banach space, \cite[Proposition 1.34]{HMVZ}. Inspection of this proof shows that, if $\{e_k;\p_k\}_{k=1}^\infty$ is an M-basis for $\X$, then there is another M-basis for $\X$ with the same linear spans and that fails to be strong. Starting from a norming M-basis shows that every separable Banach space admits a norming M-basis that is not strong. We will see an example of the opposite situation in Remark \ref{rmk: strong not norming}.

%-------------------------------------------------------%
%                                                       %
% 				CLASSES AND M-BASES   					%
%                                                       %
%-------------------------------------------------------%
\section{Classes of non-separable Banach spaces and M-bases}\label{sec: non-sep classes}
In this section we briefly introduce several classes of non-separable Banach spaces that can be characterised by the existence of M-bases with additional properties. We also describe characterisations in terms of the dual ball and give examples of spaces belonging to each class. This section has an auxiliary role for what follows, therefore we only state results without proof. \smallskip

The most important class for us is that of weakly compactly generated Banach spaces. A Banach space $\X$ is \emph{weakly compactly generated} (WCG, for short) if there exists a weakly compact subset $\K$ of $\X$ such that ${\rm span}(\K)$ is dense in $\X$, \cite{AmirLind}. Clearly, separable Banach spaces and reflexive ones are examples of WCG spaces. It is also easy to check that $c_0(\Gamma)$ for every index set $\Gamma$ and $L_1(\mu)$ for a finite measure $\mu$ are WCG. By definition, it is clear that WCG spaces are closed under taking quotients and, more generally, if $\X$ is WCG and $T\colon \X\to \Z$ is a linear operator with dense range, then $\Z$ is also WCG. However, solving an important open problem, Rosenthal gave an example of a non WCG subspace of a WCG Banach space, \cite{Rosenthal}.

WCG Banach spaces admit a characterisation in terms of M-bases as follows: a Banach space $\X$ is WCG if and only if it admits a weakly compact M-basis, \emph{i.e.}, an M-basis $\{e_\alpha; \p_\alpha\}_{\alpha\in \Gamma}$ such that $\{e_\alpha\}_{\alpha\in \Gamma}\cup \{0\}$ is weakly compact in $\X$ (see \cite[Theorem 6.9]{HMVZ}). It is easy to reformulate this property in terms of the evaluation map $\ev$: the M-basis is weakly compact if and only if
\[ \ev[B_{\X^*}]\subseteq c_0(\Gamma)\cap[-1,1]^\Gamma. \]
In particular, the dual $\X^*$ of a WCG Banach space admits a $w^*$-to-$w$ continuous injection in $c_0(\Gamma)$, \cite[Proposition 2]{AmirLind}. This directly gives that $(B_{\X^*},w^*)$ is homeomorphic to a weakly compact subset of $c_0(\Gamma)$.

A compact topological space $\K$ is \emph{Eberlein} if it is homeomorphic to a weakly compact subset of a Banach space. By the main result of \cite{AmirLind}, this is equivalent to requiring that it is a weakly compact subset of $c_0(\Gamma)$. In particular, we obtain from the previous paragraph that $(B_{\X^*},w^*)$ is Eberlein when $\X$ is WCG. However, Eberlein compacta are stable under continuous images \cite{BRW}, hence $(B_{\X^*},w^*)$ is also Eberlein when $\X$ is a subspace of a WCG Banach space (which is a more general condition, by Rosenthal's example \cite{Rosenthal}).

When $\X$ is a $\C(\K)$ space, the situation is more symmetric: in fact, $\C(\K)$ is WCG if and only if $\K$ is Eberlein, if and only if $(B_{\C(\K)^*},w^*)$ is Eberlein, \cite{AmirLind}. This and the fact that $\X$ is a subspace of $\C(B_{\X^*})$ imply that, when $B_{\X^*}$ is Eberlein, $\X$ is a subspace of the WCG space $\C(B_{\X^*})$. Therefore the dual ball of $\X$ is Eberlein if and only if $\X$ is a subspace of a WCG space. \smallskip

At one point, we will need a subclass of the class of WCG spaces, namely Hilbert generated Banach spaces and the corresponding class of uniform Eberlein compacta. A Banach space $\X$ is \emph{Hilbert generated} if there are a Hilbert space $\mathcal{H}$ and a linear map $T\colon \mathcal{H}\to \X$ with dense range. By definition, separable Banach spaces, $c_0(\Gamma)$ and $L_1(\mu)$ for a finite measure $\mu$ are Hilbert generated. However, there are reflexive Banach spaces that are not Hilbert generated \cite{KuTr} and even super-reflexive ones, \cite[Section 4]{FGHZ}. Since Hilbert spaces are WCG, every Hilbert generated Banach space is WCG. Dually, a compact topological space $\K$ is \emph{uniform Eberlein} if it is homeomorphic to a weakly compact subset of a Hilbert space. Analogous results as those of the previous paragraphs give that $\X$ is Hilbert generated if and only if
\[ \ev[B_{\X^*}]\subseteq \ell_2(\Gamma)\cap[-1,1]^\Gamma \]
and that $\X$ is a subspace of a Hilbert generated space if and only if $(B_{\X^*},w^*)$ is uniform Eberlein. Moreover, a $\C(\K)$ space is Hilbert generated if and only if $\K$ is uniform Eberlein, if and only if $(B_{\C(\K)^*},w^*)$ is uniform Eberlein, \cite[Theorem 3.2]{BRW}. For further details we refer to \cite[Section 6.3]{HMVZ} and \cite{BRW, FGHZ, FGZ}. \smallskip

The next class we wish to introduce is the class of weakly Lindel\"of determined Banach spaces, which extends the WCG class. A Banach space is \emph{weakly Lindel\"of determined} (WLD, for short) if it admits an M-basis $\{e_\alpha; \p_\alpha\}_{\alpha\in \Gamma}$ that \emph{countably supports} the dual in the sense that for every $\p\in \X^*$ 
\[ {\rm supp}(\p)\coloneqq \{\alpha\in \Gamma\colon \langle\p, e_\alpha\rangle \neq0\}\;\; \mbox{is at most countable.} \]
In this case, it turns out that every M-basis countably supports $\X^*$, \cite[Theorem 5.37]{HMVZ}. Comparing this notion to the M-basis characterisation of WCG spaces one easily sees that every WCG Banach space is WLD. Moreover, WLD spaces are closed under quotients and subspaces, as we will mention below (thus subspaces of WCG are also WLD). It is not easy to offer natural examples of WLD spaces that are not WCG; one could mention for instance Rosenthal's non WCG subspace of a WCG space \cite{Rosenthal}, or Talagrand's example of a WLD space that is not subspace of a WCG space, \cite{Talagrand WCD}.

By definition, the dual ball of a WLD space embeds in the \emph{$\Sigma$-product}
\[ \Sigma(\Gamma)=\{x\in [-1,1]^{\Gamma}\colon |\{\gamma\in \Gamma\colon x(\gamma)\neq 0\}|\leq \omega\}. \]
Therefore, $(B_{\X^*}, w^*)$ satisfies the following definition. A compact topological space $\K$ is \emph{Corson} if it is homeomorphic to a subset of $\Sigma(\Gamma)$, for some set $\Gamma$. Since elements in $c_0(\Gamma)$ are countably supported, every Eberlein compact space is Corson. Moreover, Corson compacta are closed under subspaces and continuous images, \cite{Gruenhage, Gulko, MR Eberlein+Corson}. As it turns out, a Banach space is WLD if and only if its dual ball is Corson, \cite[Section 1]{ArMe WLD}, \cite[Theorem 1.1]{VWZ} and references therein. Therefore, we get in particular that WLD spaces are closed under quotients and subspaces. Finally, a $\C(\K)$ space is WLD if and only if $\K$ is Corson and it has property (M), namely every measure on $\K$ has separable support, \cite{AMN}. \smallskip

At one point (in Theorem \ref{thm: shrinking M-basis}) we will need one more class, intermediate between WCG and WLD Banach spaces. In order not to overload the section with definitions we prefer not to enter its definition and only give some references. The class consists of the so-called \emph{weakly countably determined} (WCD, for short), or \emph{Va\v{s}\'ak}, Banach spaces; as we said, we have the implications WCG $\implies$ WCD $\implies$ WLD. WCD spaces were introduced in \cite{Vasak}, see \emph{e.g.}, \cite[Section VI.2]{DGZ}, or \cite[p. 174 and p. 433]{GMZ renorm}. \smallskip

The last generalisation that we will need is that of Plichko spaces and Valdivia compacta. An M-basis $\{e_\alpha; \p_\alpha\}_{\alpha\in \Gamma}$ is \emph{countably $\lambda$-norming} if the set of countably supported functionals
\[ \{\p\in \X^*\colon {\rm supp}(\p) \mbox{ is at most countable}\} \]
is a $\lambda$-norming subspace for $\X$. A Banach space is \emph{$\lambda$-Plichko} if it admits a countably $\lambda$-norming M-basis. By definition, every WLD Banach space is $1$-Plichko. Moreover, it is easy to find examples of Plichko Banach spaces that are not WLD, for example $\ell_1(\Gamma)$ for $\Gamma$ uncountable. One more important example is $\C([0,\omega_1])$, as we will see in Remark \ref{rmk: C(omega1) Plichko}; we refer, \emph{e.g.}, to \cite{CCS projection, DHR fundamental, Kalenda survey} for more examples of Plichko spaces. The example of $\ell_1(\Gamma)$ immediately clarifies that Plichko spaces are not closed under quotients; Kubi\'s \cite{Kubis subspace} proved that they also fail to be closed under subspaces, see Section \ref{sec: subspace}.

We now describe the topological counterpart to Plichko spaces. A compact space $\K$ is \emph{Valdivia} if there is a homeomorphic embedding $h\colon \K \to[-1,1]^\Gamma$ such that $h^{-1}(\Sigma(\Gamma))$ is dense in $\K$; in this case, the set $\Sigma(\K)\coloneqq h^{-1}(\Sigma(\Gamma))$ is called a \emph{$\Sigma$-subset} of $\K$. Clearly, every Corson compact space is Valdivia; moreover the ordinal interval $[0,\omega_1]$ is Valdivia, as witnessed by the embedding $\alpha\mapsto \bone_{[0,\alpha)}$. It is easy to see that $[0,\omega_1]$ is actually not Corson, because the point $\omega_1$ is not the limit of a sequence in $[0,\omega_1)$.

It is clear that a functional $\p\in \X^*$ is countably supported if and only if $\ev(\p)\in \Sigma(\Gamma)$. Hence if a Banach space $\X$ is $1$-Plichko, the set of countably supported functionals in $B_{\X^*}$ is $w^*$-dense in the dual ball by \eqref{eq: norming in dual ball} and the map $\ev$ witnesses that $(B_{\X^*}, w^*)$ is Valdivia. However, it is not true that the dual ball of $\X$ is Valdivia when $\X$ is merely Plichko, \cite{Kalenda renorm Valdivia}, see Remark \ref{rmk: C(omega1) Plichko}. Further, it is also not true that $\X$ must be $1$-Plichko if $(B_{\X^*}, w^*)$ is Valdivia, \cite{Kalenda Valdivia non 1Plichko}. It seems to be still open if $\X$ must be Plichko under this assumption, \cite{Kalenda Valdivia non 1Plichko}. Finally, when $\X$ is a $\C(\K)$ space, $\C(\K)$ is $1$-Plichko for $\K$ Valdivia, \cite{Orihuela} (see \cite[Proposition 5.1]{Kalenda survey}), while the converse fails, \cite{BaKu}, \cite[Section 4]{Kalenda natural}. Moreover, if $\K$ has a dense set of $G_\delta$ points, $\K$ is Valdivia if and only if $\C(\K)$ is $1$-Plichko, if and only if $(B_{\C(\K)^*},w^*)$ is Valdivia, \cite[Theorem 5.3]{Kalenda survey}. \smallskip

Before we conclude this discussion, let us mention one point concerning $\Sigma$-products. Frequently, for instance in \cite{Kalenda survey}, the $\Sigma$-product $\Sigma(\Gamma)$ is defined as
\[ \{x\in \R^{\Gamma}\colon|\{\gamma\in \Gamma\colon x(\gamma)\neq 0\}|\leq \omega\}. \]
However, every compact subset of the latter is coordinatewise bounded, hence it is contained in some product $\prod_{\alpha\in \Gamma}[-a_\alpha,a_\alpha]$, for some $a_\alpha>0$. Obviously, there is a homeomorphism between $\prod_{\alpha\in \Gamma}[-a_\alpha,a_\alpha]$ and $[-1,1]^\Gamma$ that acts by coordinates and fixes $0$ in every coordinate. Therefore using the $\Sigma$-product of lines or of the interval $[-1,1]$ does not affect the definitions. By a similar reason, one can also replace $[-1,1]^\Gamma$ with $[0,1]^\Gamma$, by using the obvious homeomorphism between $[-1,1]$ and $\big([0,1]\times\{0\}\big) \cup \big(\{0\}\times [0,1]\big)$. Defining Corson and Valdivia compacta via the $\Sigma$-product $[0,1]^\Gamma$ is also rather common. \smallskip

To summarise, we have the following implications, none of which can be reversed:
\[ \mbox{Hilbert generated} \implies \mbox{WCG} \implies \mbox{Subspace of WCG} \implies \mbox{WLD} \implies \mbox{Plichko}. \]

In the last part of this section, we briefly mention some properties of Asplund spaces. The class of Asplund spaces is transversal to the previous ones and it does not constitute a generalisation of separable Banach spaces; yet, it seems most appropriate to recall it here.  The original definition is that a Banach space $\X$ is \emph{Asplund} if every convex continuous function defined on a convex open subset $\mathcal{U}$ of $\X$ is Fr\'echet differentiable on a dense $G_\delta$ subset of $\mathcal{U}$. However, it is nowadays customary to define a Banach space $\X$ to be Asplund if every separable subspace of $\X$ has a separable dual. This equivalent definition has the advantage to be the shortest to formulate and to directly give several properties and examples. For instance, it readily implies that Asplund spaces are closed under taking closed subspaces and quotients. Moreover, it shows that every reflexive Banach space, as well as $c_0(\Gamma)$ for all set $\Gamma$, is an example of an Asplund space; on the other hand, $\ell_1(\Gamma)$ is an easy example of Banach space that is not Asplund. The equivalence between the above definitions can be found for example in \cite[Theorem I.5.7]{DGZ}.

When considering $\C(\K)$ spaces, a $\C(\K)$ space is Asplund if and only if the compact $\K$ is scattered, see, \emph{e.g.}, \cite[Lemma VI.8.3]{DGZ}. Recall that a compact space $\K$ is \emph{scattered} if every closed subset of $\K$ admits an isolated point. In particular, a very important example for us is that $\C([0,\omega_1])$ is an Asplund space. For further information and historical background on Asplund spaces we refer to \cite[\S~1.5]{DGZ}, \cite{Fabian book}, \cite{GMZ renorm}, \cite{Phelps book}, and the references therein.

%-------------------------------------------------------%
%                                                       %
% 					NORMING VS WCG    					%
%                                                       %
%-------------------------------------------------------%
\section{Norming M-bases and WCG spaces}\label{sec: WCG}
This section is dedicated to the connection between norming M-bases and WCG Banach spaces. We begin by reviewing some classical results that hinted at a close connection between these concepts and state two classical problems concerning these notions. In the second part of the section we explain some ingredients of the recent solutions to these two problems, obtained in \cite{Hajek} and \cite{HRST} respectively. \smallskip

In order to explain the connection between WCG Banach spaces and norming M-bases, we need two notions, that of \emph{projectional resolution of the identity} (PRI for short) and that of \emph{locally uniformly rotund} (LUR, for short) norm. A PRI in a Banach space $\X$ is a well-ordered collection of norm-one projections on $\X$ onto smaller subspaces, that are moreover compatible and that satisfy a certain continuity condition. For the precise definition we refer the reader to \cite[Section VI.1]{DGZ}, or \cite[Section 6.1.4]{GMZ renorm}. For us it will be more useful to keep in mind the idea that a PRI is way to decompose a Banach space into smaller well-behaved pieces, from which information on the space can be reconstructed, arguing by transfinite induction on the density character of the space. PRI's were first constructed by Lindenstrauss \cite{Lind refl} in every reflexive Banach space and shortly after by Amir and Lindenstrauss in every WCG Banach space \cite{AmirLind}. Since then, their use in non-separable Banach spaces has been pervasive and their usefulness cannot be overstated. Shortly after their introduction, Troyanski obtained his celebrated renorming theorem \cite{Troyanski LUR} that every WCG Banach space admits an LUR norm (see also Zizler's extension \cite{Zizler LUR}). Let us just say here that, besides being among the most important notions in renorming, LUR norms are also extremely useful outside the isomorphic theory; for instance they are used in Kadets' proof that all separable Banach spaces are homeomorphic, see \emph{e.g.}, \cite[Section 9.5.3]{GMZ renorm}.

Subsequently, John and Zizler \cite{JZ norming WCG} constructed a PRI in every Banach space with $1$-norming M-basis and deduced that Banach spaces with norming M-basis admit an LUR renorming. It was then natural to conjecture a connection between WCG spaces and norming M-bases; the most optimistic hope that a Banach space is WCG if and only if it admits a norming M-basis is however too naive due to the example of $\ell_1(\Gamma)$. The validity of the converse implication was asked in the same paper \cite{JZ norming WCG} and eventually became one of the classical problems in Banach space theory that was reiterated multiple times in several articles and books, \cite[Problem 111]{GMZ open}, \cite[p. 173]{HMVZ}, \cite{Kalenda omega2}, \cite[p. 108]{Rosenthal}, or \cite[Problem 10]{Zizler}.

\begin{problem}[John--Zizler]\label{probl: John-Zizler} Does every WCG Banach space admit a norming M-basis?
\end{problem}

\begin{remark}[\cite{JZ Some notes WCG}]\label{rmk: norming equiv for C(K)} Notice that in Problem \ref{probl: John-Zizler} it is sufficient to consider $\C(\K)$ spaces, in the sense that a positive answer for all WCG $\C(\K)$ spaces would imply a positive answer for all WCG spaces. Indeed, if $\X$ is any WCG Banach space, then $(B_{\X^*}, w^*)$ is Eberlein, hence $\C(B_{\X^*})$ is WCG as well, as we saw in Section \ref{sec: non-sep classes}. Moreover, $\X$ is a subspace of $\C(B_{\X^*})$. As we will see in Theorem \ref{thm: norming in WLD subspace} below, a WCG subspace of a Banach space with $\lambda$-norming M-basis admits such an M-basis as well. Thus, the existence of a norming M-basis in $\C(B_{\X^*})$ would imply the existence of a norming M-basis in $\X$.
\end{remark}

While the problem remained open until recently (see Theorem \ref{thm: H WCG not norming} below), with the advancing of the theory the connection mentioned above was clarified by the theory of WLD and Plichko spaces. Indeed both WCG spaces and spaces with norming M-basis are Plichko and Plichko spaces admit a PRI and an LUR norm, see \cite{HMVZ}.

Moreover, several partial results were given to show that in general a WCG Banach space need not have a $\lambda$-norming M-basis, for $\lambda$ close to $1$, see \cite[pp. 173--175]{HMVZ}. For example, Valdivia \cite{Valdivia 1 norming} proved that if $\X^*$ is a non-separable dual WCD space, there is an equivalent norm on $\X^* \oplus \ell_1$ that admits no $\lambda$-norming M-basis for $\lambda$ close to $1$. He also constructed a scattered compact $\K$ and a norm $\nn\cdot$ on $\C(\K)$ such that $(\C(\K),\nn\cdot)$ has the same property. It is proved in \cite[Theorem 4.4]{VWZ} that if $JL(2)$ is the Johnson--Lindenstrauss space \cite{JL} and $\n$ is an LUR norm on $JL(2)$, then $(JL(2)^*,\n^*)$ is a WCG space with no $1$-norming M-basis. This gives the existence of an Eberlein compact $\K$ such that $\C(\K)$ does not admit a $1$-norming M-basis, \cite[Corollary 4.7]{VWZ}. Moreover, Godefroy \cite[Theorem 2.1]{Godefroy Rocky} constructed a Banach space $\X$ obtained as a direct sum of a reflexive space (actually, $\ell_2(\mathfrak{c})$) and a separable space, that admits no $\lambda$-norming M-basis, for $\lambda>1/2$. Passing to the dual ball of $\X$, we obtain a (uniform) Eberlein compact space $\K$ such that $\C(\K)$ does not admit $\lambda$-norming M-basis, for $\lambda>1/2$. However, Godefroy's approach cannot work for $\lambda<1/4$, \cite[Proposition 2.4]{Godefroy Rocky}. \smallskip

As we said, the validity of the converse implication in Problem \ref{probl: John-Zizler} is easily disproved. However, several results since the Seventies strongly hinted towards the fact that the converse might be true under the additional assumption that the space is Asplund (see Problem \ref{probl: Godefroy} below). As a first instance let us mention the result from \cite[Proposition~3]{JZ norming WCG} that a Banach space $(\X,\n)$ with a $1$-norming M-basis and such that $\n$ is Fr\'echet differentiable is WCG. Since Fr\'echet differentiability of $\n$ implies that $\X$ is Asplund, this is a particular case of the above situation. Notice that there are two isometric requirements on the same norm, so the assumptions are quite restrictive and the proof is actually easy. Indeed, it is standard to check that if $\n$ is Fr\'echet differentiable, the unique $1$-norming subspace is $\X^*$ \cite[Section 3.5]{GMZ renorm}; hence the $1$-norming M-basis is automatically shrinking, thus $\X$ is WCG. A result in a similar spirit can be found in \cite{Fabian Gateaux PRI}.

As we will see, essentially the role of the Asplund property is to provide information on the dual space as well. To explain this, let us start with a result characterising spaces with a shrinking M-basis (the precise attribution to the result will be given below).

\begin{theorem} \label{thm: shrinking M-basis} For a Banach space $\X$, the following are equivalent:
\begin{romanenumerate}
    \item \label{item: shrink} $\X$ admits a shrinking M-basis;
    \item \label{item: WCG} $\X$ is WCG and Asplund;
    \item \label{item: WCD} $\X$ is WCD and Asplund;
    \item \label{item: WLD} $\X$ is WLD and Asplund;
    \item \label{item: dual LUR} $\X$ is WLD and $\X^*$ has a dual LUR norm;
    \item \label{item: Frechet} $\X$ is WLD and it admits a Fr\'echet smooth norm.
\end{romanenumerate}
\end{theorem}

Before we explain the significance of this theorem, let us quote in passing two related results. First, every Asplund generated WLD Banach space of density at most $\omega_1$ is WCG, \cite{FGZ Asplund gen}; hence the WCG and WLD classes coincide also for Asplund generated spaces (but it is open if the same result holds for higher densities). Moreover, a sufficient condition for the validity of the conditions in Theorem \ref{thm: shrinking M-basis} is that both $\X$ and $\X^*$ are subspaces of a WCG space, \cite{JZ heredity}. Previous weaker results assuming both spaces to be WCG were obtained in \cite{JZ X X*, JZ WCG duality, JL}. \smallskip

We now discuss the contributions that eventually led to the characterisation in Theorem \ref{thm: shrinking M-basis}. To begin with, the implications \eqref{item: dual LUR}$\implies$\eqref{item: Frechet}$\implies$\eqref{item: WLD} are standard \cite{DGZ} (and do not require the WLD assumption), while \eqref{item: WCG}$\implies$\eqref{item: WCD}$\implies$\eqref{item: WLD} is obvious. Next, \eqref{item: shrink}$\implies$\eqref{item: WCG} is elementary and we shall explain it here (see also the remarks after \cite[Definition 6.2.3]{Fabian book}). If $\{e_\alpha; \p_\alpha\}_{\alpha\in \Gamma}$ is a shrinking M-basis for $\X$ and $\|e_\alpha\|=1$, then $\{e_\alpha\}_{\alpha\in \Gamma}\cup\{0\}$ is weakly compact: in fact, to check that every subsequence of $\{e_\alpha\}_{\alpha\in \Gamma}$ converges weakly to $0$ it is enough to test on the linearly dense set $\{\p_\alpha\}_{\alpha\in \Gamma}$, for which functionals the convergence is obvious. Hence, $\X$ is WCG. Next, if $\Z$ is a separable subspace of $\X$, take a countable set $\Gamma_0\subseteq \Gamma$ with $\Z\subseteq \overline{\rm span}\{e_\alpha\}_{\alpha\in \Gamma_0}$; then $\Z^*\subseteq \overline{\rm span}\{\p_\alpha\cut_\Z\}_{\alpha\in \Gamma_0}$. In fact, take a functional $\p\in \Z^*$, extend it to a functional on $\X$, and approximate the extension by a linear combination $\sum_{\alpha\in \Gamma} c_\alpha \p_\alpha$; then $\sum_{\alpha\in \Gamma} c_\alpha \p_\alpha\cut_\Z\in {\rm span}\{\p_\alpha\cut_\Z\}_{\alpha\in \Gamma_0}$ approximates $\p$ (note that $\p_\alpha\cut_\Z=0$ for $\alpha\notin \Gamma_0$).

Now, to the essential contributions. Fabian \cite{Fabian dual LUR} proved that a WCD Asplund space is actually WCG and its dual admits a dual LUR norm. Inspection of the proof shows that Fabian built a shrinking M-basis and derived the WCG property out of it. Hence, \eqref{item: shrink}, \eqref{item: WCG}, and \eqref{item: WCD} are equivalent; moreover, they are also equivalent to \eqref{item: WLD}, \eqref{item: dual LUR}, and \eqref{item: Frechet} if one replaces WLD with WCD in them. The last result needed to close the circle of implications is that WLD Asplund spaces admit a shrinking M-basis, that was proved by Valdivia \cite[Proposition 2]{Valdivia}. Let us also point out the earlier result due to John and Zizler \cite[Theorem~1]{JZ Smooth WCG} that $\X$ has a shrinking M-basis if and only if $\X$ is WCG and admits a Fr\'echet smooth norm, if and only if $\X$ is WCG and $\X^*$ admits a dual LUR norm. We refer to \cite[Theorem 6.3]{HMVZ}, or \cite[Theorem 8.3.3]{Fabian book} for a proof of Theorem \ref{thm: shrinking M-basis}. \smallskip

Next, we come to the role of Theorem \ref{thm: shrinking M-basis} for norming M-bases. Fabian's argument amounts to working in $\X$ and in $\X^*$ simultaneously to build a PRI in $\X$ such that the adjoints of the projections form a PRI in $\X^*$ (let us call a PRI with this property a \emph{shrinking PRI}). Then this combines with a result of Va\v{s}\'ak \cite[Theorem 2]{Vasak} who built a shrinking M-basis from a shrinking PRI. Similarly, Valdivia built a shrinking PRI in WLD Asplund spaces, \cite[Lemma 2]{Valdivia}. Moreover, at the same time of \cite{Fabian dual LUR}, Fabian and Godefroy \cite{FG dual Asplund} built a PRI in the dual of every Asplund space. Therefore, the Asplund assumption gives a PRI in the dual, while the WCD one gives a PRI in the space. Fabian's method consisted in merging the two techniques to obtain compatible projections in the space and in the dual simultaneously (see the explanation in \cite[Chapter 8]{Fabian book} or \cite[Sections VI.2--VI.4]{DGZ}). However, John and Zizler \cite{JZ norming WCG} built a PRI in presence of a $1$-norming M-basis in $\X$. Therefore, it was natural to conjecture that one could also glue the techniques from John--Zizler \cite{JZ norming WCG} with these from Fabian--Godefroy \cite{FG dual Asplund} to build a shrinking PRI in Asplund spaces with a norming M-basis. That it were the case was indeed conjectured by Godefroy at the times when \cite{DGZ} was in preparation and the problem was subsequently recorded in various articles and books, see, \emph{e.g.}, \cite{AP}, \cite[p.~211]{HMVZ}, \cite[Problem~112]{GMZ open}.

\begin{problem}[Godefroy]\label{probl: Godefroy} Let $\X$ be an Asplund Banach space with a norming M-basis. Must $\X$ be WCG? 
\end{problem}

Before we move to explaining the solutions to Problem \ref{probl: John-Zizler} and Problem \ref{probl: Godefroy}, let us mention that Fabian \cite{Fabian dual LUR} and Fabian--Godefroy \cite{FG dual Asplund} both exploited Jayne--Rogers selectors \cite{JR selector} and Simons inequality \cite{Godefroy bdry, Simons}. For a modern approach to these results, let us refer to \cite{CF dual Asplund}, where a projectional skeleton is built in the dual of every Asplund space, and to \cite{CF skeleton} where a shrinking projectional generator is constructed in Asplund WCG spaces. \smallskip

Let us now move to the solution to Problem \ref{probl: John-Zizler}, that has been recently solved in the negative by the first named author \cite{Hajek}, in the form of the following result.
\begin{theorem}[\cite{Hajek}]\label{thm: H WCG not norming} There exists a uniform Eberlein compactum $\K$ such that the Banach space $\C(\K)$ does not admit any norming M-basis.
\end{theorem}

\begin{remark} As we mentioned in Remark \ref{rmk: norming equiv for C(K)}, the fact that the counterexample is a $\C(\K)$ space is not a strengthening of the result. On the other hand, the counterexample is even Hilbert generated, by the discussion from Section \ref{sec: non-sep classes}. This is indeed a stronger result, as, for instance, there are Eberlein compacta that are not uniform Eberlein, \cite{BeSt, Mar}.
\end{remark}

We now explain some ingredients in the proof of Theorem \ref{thm: H WCG not norming}. First of all, it is enough to build a uniform Eberlein compact space $\K_n\subseteq \ell_2(\omega_1)$ such that $\C(\K_n)$ does not admit a $1/n$-norming M-basis. Indeed, up to a scaling, one can assume that $\K_n\subseteq \frac{1}{n} B_{\ell_2(\{n\}\times \omega_1)} \subseteq \ell_2(\N\times \omega_1)$ and that $\K_n$ contains the origin. Then 
\[ \K = \bigcup_{n\in\N} \K_n\subseteq \ell_2(\N\times \omega_1) \]
is easily seen to be $w$-sequentially compact, hence $w$-compact. Moreover, there exists a retraction from $\K$ onto each $\K_n$, thus $\C(\K_n)$ is isometrically a subspace of $\C(\K)$. By Theorem \ref{thm: norming in WLD subspace} below, that we quoted already, the existence of a $\lambda$-norming M-basis in $\C(\K)$ would yield the existence of a $\lambda$-norming M-basis in $\C(\K_n)$ for each $n$, which is false for $n>1/\lambda$. Therefore, it is enough to prove the following theorem.

\begin{theorem}\label{thm: K for no norming} For every $n\in \N$ there exists a uniform Eberlein compactum $\K_n\subseteq \ell_2(\omega_1)$ such that $\C(\K_n)$ does not admit a $1/n$-norming M-basis.
\end{theorem}

\begin{proof} Fix $N\in \N$; we will explain the construction of a compact space $\K\subseteq \ell_2(\omega_1)$ such that the existence of a $\lambda$-norming M-basis in $\C(\K)$ forces $\lambda \leq \frac{1}{N-1}$ (so that $\K$ actually serves as $\K_{n-2}$). We will limit ourselves to constructing $\K$ and refer to \cite{Hajek} for the claim concerning norming M-bases.

For each choice of $\alpha_1,\dots, \alpha_N\in \omega_1$ we build a vector $z_{\alpha_1,\dots, \alpha_N}\in B_{\ell_2(\omega_1)}$ that consists of $N$ blocks. Within each block, the last coordinate is `large' in value, while the remaining coordinates are part of a sequence that goes to $0$ fast enough (and belongs to a sequence of Cantor sets). Moreover, the values of such sequences are disjoint for all blocks and for all vectors $z_{\alpha_1,\dots, \alpha_N}$ and also the positions of the blocks differ for each $z_{\alpha_1,\dots, \alpha_N}$. In the notation below, the ordinals $s_1(\alpha_1)< s_2(\alpha_1, \alpha_2)< \dots< s_N(\alpha_1,\dots, \alpha_N)$ correspond to the positions of the blocks, the functions $H(\alpha_1,\dots ,\alpha_k)$ ($k=1,\dots,N$) to the values of the sequences, and the functions $T(\alpha_1, \dots, \alpha_k)$ to the position of the sequence within the corresponding block. Now to the details.

For every $j\in \N$ pick a set $D_j$, homeomorphic to the Cantor ternary set, such that
\[ D_j\subseteq \left(\frac{1}{2^{2N+2j+2}}, \frac{1}{2^{2N+2j+1}}\right). \]
For each $k\leq N$ and each choice of $\alpha_1,\dots, \alpha_k\in \omega_1$ choose a countable set
\[ H(\alpha_1,\dots, \alpha_k)=\{t_j\}_{j=1}^\infty\in \bigcup_{j=1}^\infty D_j \]
with $t_j\in D_j$ for all $j$ and such that the sets $H(\alpha_1,\dots, \alpha_k)$ are mutually disjoint. This is easily done by transfinite induction: once all sets $H(\alpha_1,\dots, \alpha_k)$ have been chosen for all $\alpha_1,\dots, \alpha_k\leq \beta$, only countably many points have been taken in each $D_j$, that has cardinality of the continuum. Hence there is room for the next choice.

Now we choose the blocks. Take mappings $s_k\colon {\omega_1}^k\to \omega_1$, $k=1,\dots, N$ such that
\[ \eta\mapsto s_k(\alpha_1,\dots, \alpha_{k-1}, \eta)\;\; \mbox{ is strictly increasing}\]
for all $\alpha_1,\dots, \alpha_{k-1}$ and such that for $k\geq 2$
\[ s_{k-1}(\alpha_1,\dots, \alpha_{k-1})< s_k(\alpha_1,\dots, \alpha_{k-1}, \eta)\;\; \mbox{ for all } \eta. \]
Once more, these mappings are built by induction: $s_1$ is just a strictly increasing map from $\omega_1$ to itself. Then for fixed $\alpha_1$ choose a strictly increasing map $s_2(\alpha_1, \cdot)\colon \omega_1 \to (s_1(\alpha_1), \omega_1)$, and so on.

Finally, take injections $T(\alpha_1)\colon [0,s_1(\alpha_1))\to H(\alpha_1)$ and, for $k\geq 2$,
\[ T(\alpha_1,\dots, \alpha_k)\colon (s_{k-1}(\alpha_1,\dots, \alpha_{k-1}), s_k(\alpha_1,\dots, \alpha_k))\to H(\alpha_1,\dots, \alpha_k). \]

We are now in position to define the vectors $z_{\alpha_1,\dots, \alpha_N}$:
\[ z_{\alpha_1,\dots, \alpha_N}(\eta)\coloneqq \begin{cases}
    0& \mbox{if } \eta> s_N(\alpha_1,\dots, \alpha_N),\\
    \frac{1}{2^{k+1}}& \mbox{if } \eta= s_k(\alpha_1,\dots, \alpha_k),\\
    T(\alpha_1,\dots, \alpha_k)(\eta)& \mbox{if } s_{k-1}(\alpha_1,\dots, \alpha_{k-1})< \eta< s_k(\alpha_1,\dots, \alpha_k).\\
\end{cases} \]
A simple computation shows that $z_{\alpha_1,\dots, \alpha_N}\in B_{\ell_2(\omega_1)}$. Moreover, the vectors $z_{\alpha_1,\dots, \alpha_N}$ are a discrete set in the weak topology, because $z_{\alpha_1,\dots, \alpha_N}(\eta)> \frac{1}{2^{2N}}$ forces $\eta= s_k(\alpha_1,\dots, \alpha_k)$ for some $k=1,\dots, N$ and the ordinals $s_1(\alpha_1), s_2(\alpha_1, \alpha_2),\dots, s_N(\alpha_1,\dots, \alpha_N)$ identify the vector $z_{\alpha_1,\dots, \alpha_N}$ uniquely. Another important property is that if $\alpha_1= \beta_1,\dots, \alpha_k=\beta_k$, then the initial parts of the vectors coincide:
\[ z_{\alpha_1,\dots, \alpha_N} \cut_{[0,s_k(\alpha_1,\dots, \alpha_k))} = z_{\beta_1,\dots, \beta_N} \cut_{[0,s_k(\alpha_1,\dots, \alpha_k))}. \]

Finally, we can obtain the compact set $\K$, by setting
\[ \K\coloneqq \overline{\{ z_{\alpha_1,\dots, \alpha_N}\} _{\alpha_1,\dots, \alpha_N\in \omega_1}}^{\; w}. \]
The last property we wish to point out is that $\K$ is actually zero-dimensional: in fact, if $z\in \K$
\[ z(\eta) \in \bigcup_{j=1}^\infty D_j \cup \left\{0, \frac{1}{2^2}, \dots, \frac{1}{2^{N+1}}\right\}, \]
which is zero-dimensional. Then one just needs to notice that being zero-dimensional passes to subspaces and arbitrary products.
\end{proof}

We now pass to the negative solution to Problem \ref{probl: Godefroy}. The main part of the argument involves building a certain compact topological space $\K_\ro$, that we will also mention in Sections \ref{sec: C(K)} and \ref{sec: semi-Eberlein}; then the desired counterexample will be a subspace of $\C(\K_\ro)$.

The compact space will be constructed inside the power set $\mathcal{P}(\omega_1)$, that we identify as usual with the product topological space $\{0,1\}^{\omega_1}$; thus we have a compact `pointwise' topology on $\mathcal{P}(\omega_1)$. With a slight abuse of notation, we write $[\omega_1]^2=\{(\alpha,\beta)\in{\omega_1}^2\colon \alpha<\beta\}$. The main tool for the construction of the compact space is the existence of ordinal metrics on $\omega_1$, in particular the existence of a function as in the following proposition.
\begin{proposition}\label{Prop: The ro function} There exists a function $\ro\colon [\omega_1]^2 \to\omega$ such that
\begin{enumerate}[label={\rm ($\ro$\arabic*)},ref=$\ro$\arabic*]
    \item \label{item: triangle1} $\ro(\alpha,\gamma)\leq \max\{\ro(\alpha,\beta), \ro(\beta,\gamma)\}$ for every $\alpha<\beta<\gamma<\omega_1$,
    \item \label{item: triangle2} $\ro(\alpha,\beta)\leq \max\{\ro(\alpha,\gamma), \ro(\beta,\gamma)\}$ for every $\alpha<\beta<\gamma<\omega_1$.
    \item \label{item: positive} $\ro(\alpha,\beta)>0$ for all $\alpha<\beta<\omega_1$,
    \item \label{item: injective} $\ro(\alpha,\gamma)\neq\ro(\beta,\gamma)$, for all $\alpha<\beta<\gamma<\omega_1$.
\end{enumerate}
\end{proposition}

It is also convenient to add the `boundary condition' $\ro(\alpha,\alpha)=0$. Conditions \eqref{item: triangle1} and \eqref{item: triangle2} above imply that $\ro$ resembles in certain ways a metric on $\omega_1$, whence the name \emph{ordinal metric}. They were introduced by Todor\v{c}evi\'c in \cite{T Acta} and analysed in detail in \cite{T Walks}; see also \cite{Bekkali}. The construction of a function $\ro$ as in Proposition \ref{Prop: The ro function} can be found in \cite[Lemma 3.2.2]{T Walks} and some details are also in \cite[Section 3]{HRST}.

Having fixed such a function $\ro\colon [\omega_1]^2 \to\omega$, for $n<\omega$ and $\alpha<\omega_1$ we let
\[ F_n(\alpha):=\{\xi\leq\alpha \colon \ro(\xi,\alpha)\leq n\}; \]
note that $F_0(\alpha)=\{\alpha\}$, by \eqref{item: positive} above, while $|F_n(\alpha)|\leq n+1$ by \eqref{item: injective}. We are now in position to define the desired compact space as follows
$$\F_\ro:=\{F_n(\alpha)\colon n<\omega,\, \alpha<\omega_1\} \qquad\text{and}\qquad \K_\ro:=\overline{\F_\ro},$$
where the closure is intended in the pointwise topology of $\mathcal{P}(\omega_1)$. Notice that the sequence $(F_n(\alpha))_{n=1}^\infty$ converges pointwise to the set $[0,\alpha]$; in other words, the sets $F_n(\alpha)$ can be used as inner approximations of all countable ordinal intervals. We can now list the main properties of the compact space $\K_\ro$.

\begin{theorem}[{\cite[Theorem 3.5]{HRST}}]\label{thm: prop of Kro} The compact space $\K_\ro$ defined above has the following properties:
\begin{romanenumerate}
\item \label{item: Kro(i)} $\{\alpha\}\in\K_\ro$ for every $\alpha<\omega_1$,
\item \label{item: Kro(ii)} $[0,\alpha)\in\K_\ro$ for every $\alpha\leq\omega_1$,
\item \label{item: Kro(iii)} if $A\in\K_\ro$ is an infinite set, then $A=[0,\alpha)$ for some $\alpha\leq\omega_1$,
\item \label{item: Kro(iv)} $\K_\ro$ is scattered.
\end{romanenumerate}
\end{theorem}

Notice that $\K_\ro$ is the closure of a collection of finite sets, hence elements of $\K_\ro$ can be approximated by finite sets; in this respect, $\K_\ro$ resembles the Tikhonov cube $[0,1]^{\omega_1}$, or the Cantor cube $\{0,1\}^{\omega_1}$. On the other hand, $\K_\ro$ contains very few infinite sets, just the initial intervals, for which reason it is scattered. This property resembles the behaviour of $[0,\omega_1]$ (we will return to this point in Remark \ref{rmk: weak P point in Kro}).

\begin{proof} The validity of \eqref{item: Kro(i)} and \eqref{item: Kro(ii)} follows easily from what has been said before, while \eqref{item: Kro(iv)} can be deduced from \eqref{item: Kro(iii)} via a certain maximality argument, somewhat similar to the proof that Eberlein subsets of $\{0,1\}^\Gamma$ are scattered (as in \cite[Lemma 2.53]{HMVZ}).

We shall now explain the proof of \eqref{item: Kro(iii)}. (The argument that follows is shorter than the one in \cite{HRST} and was suggested to us by Chris Lambie-Hanson). Fix an infinite set $A\in\K_\ro$. Our goal is to show that
\[ \alpha\in A,\, \tilde{\alpha}<\alpha \implies \tilde{\alpha} \in A. \]
Let $k\coloneqq \ro(\tilde{\alpha},\alpha)$ and fix any neighbourhood $\mathcal{U}$ of $A$ in $\K_\ro$. Since $A$ is infinite, we can pick mutually distinct $\beta_1,\dots, \beta_k\in A \setminus\{\alpha\}$. By definition of the pointwise topology, we can find a neighbourhood $\mathcal{V}\subseteq \mathcal{U}$ of $A$ such that every $F\in \mathcal{V}$ satisfies $\alpha,\beta_1, \dots,\beta_k\in F$. Take $F_n(\beta) \in \F_\ro\cap \mathcal{V}$ (recall that $\F_\ro$ is dense in $\K_\ro$); then $\alpha,\beta_1, \dots,\beta_k\in F_n(\beta)$, whence $k+1\leq |F_n(\beta)|\leq n+1$. Moreover, $\alpha\in F_n(\beta)$ implies $\alpha\leq\beta$ and $\ro(\alpha,\beta)\leq n$. By our choice of $k$ we also have $\ro(\tilde{\alpha},\alpha)= k\leq n$, hence \eqref{item: triangle1} implies that $\ro(\tilde{\alpha},\beta)\leq n$. This yields that $\tilde{\alpha}\in F_n(\beta)$ and therefore $\tilde{\alpha} \in A$.
\end{proof}

Incidentally, let us remark that we did not need the second `triangle inequality' \eqref{item: triangle2} in the argument. Having Theorem \ref{thm: prop of Kro} at our disposal, we can now explain the construction of the Banach space that solves Problem \ref{probl: Godefroy} in the negative. The following is the main result of \cite{HRST}.
\begin{theorem}[{\cite[Theorem A]{HRST}}]\label{thm: HRST} There exists an Asplund space $\X_\ro$ with a $1$-norming M-basis such that $\X_\ro$ is not WCG.
\end{theorem}

\begin{proof} The construction is performed in the Banach space $\C(\K_\ro)$. We consider the biorthogonal system $\{f_\alpha; \mu_\alpha\}_{ \alpha<\omega_1}$ in the Banach space $\C(\K_\ro)$ defined by
\begin{equation*}\begin{split} f_\alpha \in \C(\K_\ro) \qquad \qquad & f_\alpha(A)=\begin{cases} 1 & \alpha\in A \\ 0 & \alpha \notin A \end{cases} \qquad (A\in \K_\ro) \\
\mu_\alpha:=\delta_{\{\alpha\}} \in \mathcal{M}(\K_\ro) \qquad\qquad & \mu_\alpha(S)=\begin{cases} 1 & \{\alpha\} \in S \\ 0 & \{\alpha\} \notin S \end{cases} \qquad (S\subseteq \K_\ro).
\end{split}\end{equation*}
It is not difficult to see that the biorthogonal system is well defined (for instance, that $f_\alpha$ is a continuous function) and that it is actually biorthogonal.

The desired Banach space is then $\X_\ro\coloneqq \overline{\rm span}\{f_\alpha\}_{\alpha< \omega_1}$. Because $\K_\ro$ is scattered and $\X_\ro\subseteq \C(\K_\ro)$ it is immediate that $\X_\ro$ is an Asplund space. Additionally, $\{f_\alpha; \mu_\alpha\cut_{\X_\ro}\}_{ \alpha<\omega_1}$ is a biorthogonal system in $\X_\ro$ which by definition is fundamental. Thus, we have to prove that the said system is $1$-norming and that $\X_\ro$ is not WCG.

\begin{claim}\label{claim: Xro not WCG} $\X_\ro$ is not WCG.
\end{claim}
\begin{proof}[Proof of Claim \ref{claim: Xro not WCG}] \renewcommand\qedsymbol{$\square$} It is enough to prove that the dual ball $(B_{\X^*},w^*)$ contains $[0,\omega_1]$, whence it is not Corson (this shows that $\X_\ro$ is not WLD, which by Theorem \ref{thm: shrinking M-basis} is not more general in our context). We already observed that $[0,\omega_1]$ is contained in $\K_\ro$, hence in the dual ball $B_{\mathcal{M}(\K_\ro)}$. This copy is preserved by the (restriction) quotient map onto $(B_{\X^*},w^*)$, because the functions $\{f_\alpha\}_{\alpha<\omega_1}$ separate points of $\K_\ro$.

More precisely, the map $\alpha\mapsto \delta_{[0,\alpha)}\cut_{\X_\ro}$ is a continuous injection of $[0,\omega_1]$ into $(B_{\X^*},w^*)$; injectivity follows from the fact that, if $\alpha< \beta$, $\langle \delta_{[0,\alpha)}\cut_{\X_\ro}, f_\alpha\rangle =0$ and $\langle\delta_{[0,\beta)}\cut_{\X_\ro}, f_\alpha\rangle =1$.
\end{proof}

\begin{claim}\label{claim: 1-norming} $\{f_\alpha; \mu_\alpha\cut_{\X_\ro}\}_{ \alpha<\omega_1}$ is $1$-norming.
\end{claim}
\begin{proof}[Proof of Claim \ref{claim: 1-norming}] \renewcommand\qedsymbol{$\square$} Here the main point lies in proving the crucial formula
\begin{equation}\label{eq: sum of Diracs}
    \delta_A\cut_{\X_\ro} = \sum_{\alpha\in A} \delta_{\{\alpha\}}\cut_{\X_\ro} \qquad A\in \F_\ro.
\end{equation}
In particular, it follows that each $\delta_A\cut_{\X_\ro}$ belongs to ${\rm span}\{\mu_\alpha\cut_{\X_\ro}\} _{\alpha<\omega_1}$ (and it trivially has norm one). Hence, for all $f\in \X_\ro$ we obtain 
\begin{align*}
    \|f\|=& \sup_{A\in\F_\ro}|f(A)|= \sup_{A\in\F_\ro}|\langle\delta_A\cut_{\X_\ro} ,f\rangle|\\
    \leq& \sup\left\{|\langle \mu,f\rangle|\colon \mu\in {\rm span}\{\mu_\alpha\cut_{\X_\ro}\}_{\alpha<\omega_1} , \|\mu\|\leq1 \right\},
\end{align*}
which implies that the M-basis $\{f_\alpha; \mu_\alpha\cut_{\X_\ro}\}_{ \alpha<\omega_1}$ is $1$-norming.

Therefore, we are only left with the proof of \eqref{eq: sum of Diracs}. For this, it is clearly enough to check that the two functionals agree on each $f_\beta$, $\beta<\omega_1$. However, this is immediate: if $\beta\in A$, both functionals equal $1$, if $\beta\notin A$, both vanish.
\end{proof}

In the end, the point behind the choice of the subspace $\X_\ro$ is to choose a sufficiently small subspace to achieve the validity of \eqref{eq: sum of Diracs}. Admittedly, the fact that the argument runs so smoothly still bears a bit of magic for us.
\end{proof}

To conclude our discussion of \cite{HRST}, let us point out explicitly that the counterexample in Theorem \ref{thm: HRST} is not a $\C(\K)$ space, unlikely the example of Theorem \ref{thm: H WCG not norming}. In particular, we do not know the answer to the following problem.
\begin{problem}\label{probl: norming in Kro} Does the Banach space $\C(\K_\ro)$ admit a norming M-basis?
\end{problem}

We now enter the last part of the section, where we show that Godefroy's problem, Problem \ref{probl: Godefroy}, has a positive answer if we strengthen the Asplund assumption by asking that $\X$ has countable Szlenk index. The validity of this was suggested to us by Gilles Godefroy and to the best of our knowledge the result has not appeared elsewhere.

Recall that for a $w^*$-compact subset $\K$ of $\X^*$ and $\e>0$, $s_\e(\K)$ is the set obtained by removing from $\K$ all $w^*$-open subsets of $\K$ with diameter less than $\e$. Then the derivation $s_\e^\alpha(\K)$ is defined by $s_\e^{\alpha+1}(\K)= s_\e (s_\e^\alpha(\K))$ and $s_\e^\alpha(\K)= \bigcap_{\beta< \alpha} s_\e^\beta(\K)$ for $\alpha$ a limit ordinal. If $s_\e^\alpha(\K)= \emptyset$ for some $\alpha$, one lets ${\rm Sz}(\K,\e)$ be the least such an ordinal $\alpha$; otherwise, ${\rm Sz}(\K,\e)= \infty$. Finally, the \emph{Szlenk index} of $\K$ is ${\rm Sz}(\K)\coloneqq \sup_{\e>0} {\rm Sz}(\K,\e)$ and the Szlenk index of $\X$ is ${\rm Sz}(\X)\coloneqq {\rm Sz}(B_{\X^*})$. Clearly, if $\mathcal{L}\subseteq \K$ are $w^*$-compact $s_\e^\alpha(\mathcal{L})\subseteq s_\e^\alpha(\K)$; thus ${\rm Sz}(\mathcal{L})\leq {\rm Sz}(\K)$. Moreover, it is known that ${\rm Sz}(\X)< \infty$ if and only if $\X$ is an Asplund space, \cite[Theorem I.5.2]{DGZ}. Therefore, admitting countable Szlenk index is stronger than being Asplund (notice that for an Asplund space $\X$ of density $\kappa$ one has ${\rm Sz}(\X)< \kappa^+$, \cite[Lemma 2.41]{HMVZ}). For further information on the Szlenk index we refer to \cite[Chapter 2]{HMVZ}, \cite{Lancien}, or Causey's work, \emph{e.g.}, \cite{Causey1, Causey2, Causey3, Causey4, Causey5}. \smallskip

The main ingredient in the proof is the following generalisation of a result due to Deville and Godefroy \cite{DG} that we will also mention several times in the next section. Further generalisations of Deville--Godefroy's result can be found in Kalenda's papers \cite{Kalenda image Valdivia, Kalenda Valdivia equiv norm}, \cite[Proposition 3.12]{Kalenda survey}.

\begin{proposition}\label{prop: separated omega1} Let $\X$ be a Banach space and $\K$ be a $w^*$-compact subset of $\X^*$. If $\K$ is Valdivia and not Corson, then $\K$ contains a homeomorphic copy of $[0,\omega_1]$ that is $\e$-separated (in norm), for some $\e>0$.
\end{proposition}
Notice that the assumption that $\K\subseteq (\X^*,w^*)$ does not affect the gene\-ra\-lity, as every compact set $\K$ is homeomorphic to a $w^*$-compact subset of $\C(\K)^*$.

\begin{proof} By the result of Deville and Godefroy \cite{DG} there exists a homeomorphic embedding $\p\colon [0,\omega_1]\to (\K,w^*)$. We first notice that the (disjoint) $w^*$-compact sets $\p([0,\alpha])$ and $\p([\alpha+1,\omega_1])$ have positive distance (in the norm of $\X^*$), for all $\alpha<\omega_1$. Indeed, if the distance were $0$, there would be sequences $(p_n)_{n=1}^\infty$ in $\p([0,\alpha])$ and $(q_n)_{n=1}^\infty$ in $\p([\alpha+1,\omega_1])$ with $\|p_n- q_n\|< 1/n$. Up to a subsequence, we can assume that $p_n$ is $w^*$-convergent to some $p\in \p([0,\alpha])$. Thus $q_n$ $w^*$-converges to $p$ as well, which gives the contradiction that $p\in \p([\alpha+1,\omega_1])$.

Therefore, for $\alpha< \omega_1$ we can define
\[ d_\alpha\coloneqq {\rm dist} \big( \p([0,\alpha]), \p([\alpha+1,\omega_1]) \big)>0. \]
Hence there are an uncountable subset $\Lambda$ of $[0,\omega_1]$ and $\e>0$ such that $d_\alpha \geq\e$ for all $\alpha\in \Lambda$. Moreover, $\overline{\Lambda}$ is homeomorphic to $[0,\omega_1]$ and $\Lambda$ is obviously dense in $\overline{\Lambda}$. Thus, up to replacing $\p$ with $\p\cut_{\overline{\Lambda}}$ we can assume that $\Lambda$ is dense in $[0,\omega_1]$. In particular, all successor ordinals belong to $\Lambda$. Therefore, if $\alpha<\beta\in [0,\omega_1]$ are such that there is a successor ordinal $\gamma$ with $\alpha\leq \gamma<\beta$, then $\p(\alpha)\in \p([0,\gamma])$ and $\p(\beta)\in \p([\gamma+1,\omega_1])$, whence $\|\p(\alpha)- \p(\beta)\|\geq d_\gamma\geq \e$. In other words, if $\Gamma\coloneqq [0,\omega_1]\setminus \{\alpha+1\colon \alpha \mbox{ limit ordinal}\}$, then $\Gamma$ is homeomorphic to $[0,\omega_1]$ and $\p(\Gamma)$ is $\e$-separated, as desired.
\end{proof}

\begin{remark} The above proof actually shows that if $[0,\omega_1]\subseteq \X^*$, then there is a $w^*$-compact subset $\mathcal{L}$ of $[0,\omega_1]$ that is homeomorphic to $[0,\omega_1]$ and $\e$-separated. However, this also follows from the statement of the proposition by choosing $\K= [0,\omega_1]$. 
\end{remark}

We are now ready for the proof of the claimed result. As it turns out, we can actually weaken the assumption on the norming M-basis to the one that $\X$ is Plichko.
\begin{theorem}\label{thm: countable Sz} Every Plichko space with countable Szlenk index is WCG.
\end{theorem}

\begin{proof} Being WCG and having countable Szlenk index are isomorphic properties, so there is no loss in generality in assuming that $\X$ is $1$-Plichko. Hence $(B_{\X^*},w^*)$ is Valdivia. If it is additionally Corson, then $\X$ is WLD, thus WCG by Theorem \ref{thm: shrinking M-basis}. Therefore, we assume towards a contradiction that $(B_{\X^*},w^*)$ is not Corson and we apply Proposition \ref{prop: separated omega1}. Hence, there is a $w^*$-compact set $\K\subseteq (B_{\X^*},w^*)$, homeomorphic to $[0,\omega_1]$ and that is $\e$-separated, for some $\e>0$. For an $\e$-separated set $\mathcal{L}$ it is clear that $s_\e(\mathcal{L})$ coincides with the set of accumulation points $\mathcal{L}'$ of $\mathcal{L}$; therefore, the Szlenk $\e$-derivation of $\K$ is the same as the Cantor--Bendixson one, $s_\e^\alpha(\K)= \K^{(\alpha)}$ for every ordinal $\alpha$. Thus we have 
\[ s_\e^{\omega_1}(B_{\X^*})\geq s_\e^{\omega_1}(\K)= \K^{(\omega_1)}\neq \emptyset, \]
which contradicts the assumption that $\X$ has countable Szlenk index.
\end{proof}

%-------------------------------------------------------%
%                                                       %
% 					  C(K) SPACES      					%
%                                                       %
%-------------------------------------------------------%
\section{Banach spaces of continuous functions}\label{sec: C(K)}
The motivation for this section comes from a comparison between Problem \ref{probl: John-Zizler} and Problem \ref{probl: Godefroy}. For the former we observed in Remark \ref{rmk: norming equiv for C(K)} that it suffices to consider $\C(\K)$ spaces and Theorem \ref{thm: H WCG not norming} exhibits an explicit $\C(\K)$ counterexample; for the latter, the counterexample in Theorem \ref{thm: HRST} is only a subspace of an Asplund $\C(\K)$ space. Therefore, it is natural to ask if Problem \ref{probl: Godefroy} might admit a positive answer for $\C(\K)$ spaces. This problem is still open and we will discuss it in the present section. Let us start by giving a topological reformulation to it. As we saw in Section \ref{sec: non-sep classes}, a $\C(\K)$ space is WCG if and only if $\K$ is Eberlein and it is Asplund if and only if $\K$ is scattered. Therefore, Problem \ref{probl: Godefroy} for $\C(\K)$ spaces can be restated as follows.
\begin{problem}[Godefroy's problem for $\C(\K)$ spaces]\label{probl: Godefroy for C(K)} Let $\K$ be a scattered compact space such that $\C(\K)$ admits a norming M-basis. Is $\K$ Eberlein? 
\end{problem}

Our plan for the section is the following. At first, we shall explain why consideration of the above problem naturally leads one to consider the compact space $[0,\omega_1]$. Secondly, we describe an important result due to Alexandrov and Plichko \cite{AP} that $\C([0,\omega_1])$ does not admit a norming M-basis and its recent extension \cite{RS C(K)}. Finally, we state and comment some problems whose positive answers would solve Problem \ref{probl: Godefroy for C(K)} in the positive. \smallskip

There is an obvious reason why one might think of $[0,\omega_1]$ when considering Problem \ref{probl: Godefroy for C(K)}, namely the fact that $[0,\omega_1]$ is a rather simple and well studied example of a scattered compact space that is not Eberlein. However, there is much more than that. Consider the particular case of Problem \ref{probl: Godefroy for C(K)} when the M-basis is actually $1$-norming. Then the evaluation map $\ev$ from \eqref{eq: ev map} shows that $(B_{\C(\K)^*}, w^*)$ is Valdivia, thus $\K$ is Valdivia as well, because of \cite[Theorem 5.3]{Kalenda survey} and the fact that $\K$ has a dense set of $G_\delta$ points (as it is scattered).

Now the main point comes, due to the following important results. First, Deville and Godefroy \cite{DG} proved the elegant characterisation that a Valdivia compactum $\K$ is Corson if and only if it does not contain $[0,\omega_1]$. Next, by a result of Alster \cite{Alster} scattered Corson compacta are Eberlein. Therefore scattered Valdivia compacta that do not contain $[0,\omega_1]$ are Eberlein and one only has to consider these $\K$ that contain $[0,\omega_1]$. For this reason, $[0,\omega_1]$ could be considered to be a good test case for Problem \ref{probl: Godefroy for C(K)}. \smallskip

Therefore, we shall now focus on the Banach space $\C([0,\omega_1])$. Before we discuss the answer to Problem \ref{probl: Godefroy for C(K)} for this particular case, let us briefly mention the existence of a canonical M-basis in the space.

\begin{example}\label{ex: canonical M-basis in [0,omega1]} For notational convenience, we work in the space $\{f\in \C([0,\omega_1])\colon f(0)=0\}$ (which is clearly isometric to $\C([0,\omega_1])$). For $\alpha< \omega_1$, consider
\[ f_\alpha\coloneqq \bone_{(\alpha,\omega_1]} \qquad\mbox{and}\qquad \mu_\alpha\coloneqq \delta_{\alpha+1}- \delta_\alpha \]
(noting that actually $\mu_0=\delta_1$). It is easy to realise that the system is biorthogonal; moreover, the linear span of the functions $\{f_\alpha\}_{\alpha< \omega_1}$ is dense because of the Stone--Weiestrass theorem. Finally, if $\langle \mu_\alpha,f\rangle=0$ for all $\alpha<\omega_1$, the continuity of $f$ implies that it must be constant; hence $f(0)=0$ forces $f=0$. Thus $\{f_\alpha; \mu_\alpha\}_{\alpha<\omega_1}$ is an M-basis.

Sometimes it is convenient to consider a slightly different way to express the above M-basis. Consider the hyperplane
\[ \C_0([0,\omega_1])\coloneqq \{f\in \C([0,\omega_1])\colon f(\omega_1)=0\}, \]
which is clearly isomorphic to $\C([0,\omega_1])$, and the biorthogonal system
\[ f_\alpha\coloneqq \bone_{[0,\alpha]} \qquad\mbox{and}\qquad \mu_\alpha\coloneqq \delta_\alpha- \delta_{\alpha+1} \qquad(\alpha<\omega_1). \]
The same argument as above shows that this system is an M-basis for $\C_0([0,\omega_1])$; in fact, the canonical isomorphism between $\{f\in \C([0,\omega_1])\colon f(0)=0\}$ and $\C_0([0,\omega_1])$ maps one M-basis onto the other.
\end{example}

\begin{remark}\label{rmk: C(omega1) Plichko} Let us also mention that the M-basis we constructed in $\C([0,\omega_1])$ is countably $1$-norming. In fact, the set
\[ S\coloneqq\{\mu\in \C([0,\omega_1])^*\colon \{ \alpha<\omega_1\colon \langle\mu, f_\alpha\rangle\neq 0 \} \mbox{ is countable} \} \]
of countably supported functionals coincides with $\{\mu\in \C([0,\omega_1])^*\colon \mu(\{\omega_1\})=0 \}$, as it is easy to see. However, $S$ contains all the Dirac measures $\delta_\alpha$, $\alpha<\omega_1$, hence it is $1$-norming.

On the other hand, a similar argument only shows that the above system in $\C_0([0,\omega_1])$ is countably $1/2$-norming and Kalenda \cite{Kalenda renorm Valdivia} proved that $\C_0([0,\omega_1])$ does not admit any countably $1$-norming M-basis.
\end{remark}

Let us now pass to the existence of norming M-bases in $\C([0,\omega_1])$, that was answered negatively by Alexandrov and Plichko in \cite{AP}.
\begin{theorem}[\cite{AP}]\label{thm: AP} The Banach space $\C([0,\omega_1])$ does not admit a norming M-basis.
\end{theorem}

\begin{remark}\label{rmk: long Schauder not norming} It is not hard to check explicitly that the M-basis from Example \ref{ex: canonical M-basis in [0,omega1]} is even a long Schauder basis, see, \emph{e.g.}, \cite[Proposition 4.8]{HMVZ}. Therefore, unlike the separable case, long Schauder bases might fail to be norming. An analogue example in $\C([0,\omega_2])$ shows that long Schauder bases might also fail to be countably norming, \cite{Kalenda omega2} (see the discussion concerning $\C([0,\omega_2])$ at the end of Section \ref{sec: subspace}).
\end{remark}

\begin{remark}\label{rmk: strong not norming} By definition, every long Schauder basis is a strong M-basis, therefore we have the existence of a Banach space with strong M-basis, but no norming M-basis. On the other hand, every Banach space with norming M-basis admits a strong M-basis, \cite[Theorem 1]{AP}; this is actually true more generally for every Plichko space and even for spaces with a projectional skeleton \cite[Theorem 1.2]{Kalenda PLMS} and it is a standard transfinite induction argument from Terenzi's result \cite{Terenzi} for separable spaces.
\end{remark}

Instead of giving the full proof of Theorem \ref{thm: AP}, we will follow \cite{Alexandrov} and prove that the analogue in $\C([0,\eta])$ of the M-basis from Example \ref{ex: canonical M-basis in [0,omega1]} is not norming provided that $\eta\geq \omega^2$. The idea is to consider a certain `ziqqurat like' function, see \eqref{eq: ziqqurat} below, which is also present in \cite{AP}, \cite[Theorem 5.25]{HMVZ}, \cite{Kalenda omega2}, and \cite{RS C(K)}. The proof of Theorem \ref{thm: AP} then consists of several reductions and stabilisation arguments to show that one can reproduce a function as \eqref{eq: ziqqurat} starting from any M-basis of $\C([0,\omega_1])$.

\begin{proposition}[\cite{Alexandrov}] The M-basis $\{f_\alpha; \mu_\alpha\}_{\alpha< \eta}$ of $\C_0([0,\eta])$ given by
\[ f_\alpha\coloneqq \bone_{[0,\alpha]} \qquad\mbox{and}\qquad \mu_\alpha\coloneqq \delta_\alpha- \delta_{\alpha+1} \]
is not norming when $\eta\geq \omega^2$.
\end{proposition}

\begin{proof}
For $N\in \N$ consider the function (see Figure \ref{pic: ziqqurat z4})

\begin{equation}\label{eq: ziqqurat}
    z_N\coloneqq \sum_{k=1}^N \frac{k}{N} \bone_{(\omega\cdot k,\omega\cdot (k+1)]} + \sum_{k=N+1}^{2N-1} \frac{2N-k}{N} \bone_{(\omega\cdot k,\omega\cdot (k+1)]}.
\end{equation}

\begin{figure}[!h]\begin{center}
\begin{tikzpicture}
    %axes
    \draw[->] (0,0) -- (9,0);
    \draw[->] (0,0) -- (0,2.5);

    %levels of function
    \draw[thick] (0,0) -- (1,0);
    \draw[thick] (1,0.5) -- (2,0.5);
    \draw[thick] (2,1) -- (3,1);
    \draw[thick] (3,1.5) -- (4,1.5);
    \draw[thick] (4,2) -- (5,2);
    \draw[thick] (5,1.5) -- (6,1.5);
    \draw[thick] (6,1) -- (7,1);
    \draw[thick] (7,0.5) -- (8,0.5);
    \draw[thick] (8,0) -- (9,0);

    %vertical dotted lines
    \draw[dotted] (1,0) -- (1,0.5);
    \draw[dotted] (2,0) -- (2,1);
    \draw[dotted] (3,0) -- (3,1.5);
    \draw[dotted] (4,0) -- (4,2);
    \draw[dotted] (5,0) -- (5,2);
    \draw[dotted] (6,0) -- (6,1.5);
    \draw[dotted] (7,0) -- (7,1);
    \draw[dotted] (8,0) -- (8,0.5);

    %values of alpha
    \node at (1,-0.28) {\small $\omega$};
    \node at (2,-0.28) {\small $2\omega$};
    \node at (3,-0.28) {\small $3\omega$};
    \node at (4,-0.28) {\small $4\omega$};
    \node at (5,-0.28) {\small $5\omega$};
    \node at (6,-0.28) {\small $6\omega$};
    \node at (7,-0.28) {\small $7\omega$};
    \node at (8,-0.28) {\small $8\omega$};

    %values of zN
    \node at (-0.5,0.5) {\small $1/4$};
    \node at (-0.5,1) {\small $1/2$};
    \node at (-0.5,1.5) {\small $3/4$};
    \node at (-0.28,2) {\small $1$};

    %full dots
    \fill[black] (1,0) circle (2pt);
    \fill[black] (2,0.5) circle (2pt);
    \fill[black] (3,1) circle (2pt);
    \fill[black] (4,1.5) circle (2pt);
    \fill[black] (5,2) circle (2pt);
    \fill[black] (6,1.5) circle (2pt);
    \fill[black] (7,1) circle (2pt);
    \fill[black] (8,0.5) circle (2pt);

    %empty dots
    \draw (1,0.5) circle (2pt);
    \draw (2,1) circle (2pt);
    \draw (3,1.5) circle (2pt);
    \draw (4,2) circle (2pt);
    \draw (5,1.5) circle (2pt);
    \draw (6,1) circle (2pt);
    \draw (7,0.5) circle (2pt);
    \draw (8,0) circle (2pt);

    %grid on y axis
    \foreach \x in {0, 1, 2, 3, 4, 5, 6, 7, 8} \draw (\x cm,1pt) -- (\x cm,-1pt);
    \foreach \y in {0, 0.5, 1, 1.5, 2} \draw (1pt,\y cm) -- (-1pt,\y cm);
\end{tikzpicture}
\caption{The function $z_N$ for $N=4$.}
\label{pic: ziqqurat z4}
\end{center}\end{figure}
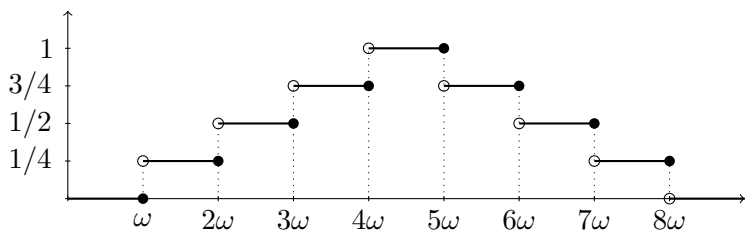

Notice that when $\alpha\notin\{\omega,\dots,\omega\cdot 2N\}$ we have $\langle\mu_\alpha, z_N\rangle=0$, while $\langle\mu_\alpha, z_N\rangle= -1/N$ for $\alpha= \omega,\dots,\omega\cdot N$ and $\langle\mu_\alpha, z_N\rangle= 1/N$ for $\alpha= \omega\cdot (N+1),\dots,\omega\cdot 2N$. Therefore, given any measure $\mu= \sum_{\alpha<\eta} c_\alpha\mu_\alpha$ in ${\rm span}\{\mu_\alpha\}_{\alpha<\eta}$ (where we understand that only finitely many scalars $c_\alpha$ are non-zero) we have
\begin{align*}
    |\langle\mu, z_N\rangle|\leq& \sum_{k=1}^{2N} \frac{1}{N} |c_{\omega\cdot k}|= \frac{1}{N}\sum_{k=1}^{2N} |c_{\omega\cdot k}|\cdot \mu_{\omega\cdot k} (\{\omega\cdot k\})\\
    =& \frac{1}{N}\sum_{k=1}^{2N} |\mu(\{\omega\cdot k\})|\leq \frac{1}{N} \|\mu\|.
\end{align*}
Since $N$ was arbitrary, ${\rm span}\{\mu_\alpha\}_{\alpha<\eta}$ is not norming.
\end{proof}

With the result for $\C([0,\omega_1])$ at our disposal, it is then natural to try to adapt the same type of argument and attack Problem \ref{probl: Godefroy for C(K)} for a general compact $\K$. This attempt was undertook in \cite{RS C(K)}, whose outcome briefly is the following: Alexandrov--Plichko's method cannot solve the problem in full generality, but it does give an answer in some interesting particular cases. The main result from \cite{RS C(K)} is the following generalisation of Theorem \ref{thm: AP}.

\begin{theorem}[\cite{RS C(K)}]\label{thm: AP-RS} $\C([0,\omega_1])$ does not embed in a Banach space with a norming M-basis.

In particular, $\C(\K)$ does not admit a norming M-basis, whenever $[0,\omega_1]$ is a continuous image of $\K$.
\end{theorem}
The second clause follows from the first one and the standard fact that if $\p\colon \K\to [0,\omega_1]$ is a continuous surjection, then the map $f\mapsto f\circ \p$ defines an isometric embedding of $\C([0,\omega_1])$ in $\C(\K)$. The advantage of Theorem \ref{thm: AP-RS} over Theorem \ref{thm: AP} will be witnessed in several instances in what follows. As a first sample, we immediately get that Banach spaces such as $\C([0,\eta])$ for $\eta\geq \omega_1$, $\C([0,\omega_1]\times [0,1])$, $\C([0,\omega_1]^2)$, or $\C([0,2\omega_1])$ do not admit norming M-bases (recall that $\C([0,2\omega_1])= \C([0,\omega_1]) \oplus \C([0,\omega_1])$ is not isomorphic to $\C([0,\omega_1])$, \cite{Semadeni}). However, see the discussion in Section \ref{sec: subspace}. \smallskip

Let us focus on the good news first. An important class for which we can answer Problem \ref{probl: Godefroy for C(K)} is the class of compact trees. Trees are an interesting class of compacta because on the one hand they are a well-behaved generali\-sation of ordinals and on the other hand this bigger flexibility already allows for several interesting examples. As instances of this let us mention some recent results in topology \cite{Nyikos, RS tree, JS tree1, JS tree2} and Haydon's celebrated results in renorming \cite{Haydon1, Haydon2, Haydon3}. Here we adopt the set-theoretic definition of tree as a poset where the set of predecessors of each element is well-ordered. There are several natural topologies that one can consider on a tree, leading to rather different behaviours, see, \emph{e.g.}, \cite{Nyikos tree}. For our purposes, the correct choice is the so-called coarse-wedge topology (while in renorming theory, as in Haydon's papers, the interval topology is the suitable one).

It is not difficult to see that if a tree $\mathcal{T}$ has height at most $\omega_1$, then it is Corson (because every element of the tree actually has countable height), \cite[Theorem 2.8]{Nyikos}. On the other hand, if the height of $\mathcal{T}$ is at least $\omega_1+1$, we can take an element $t$ with height $\omega_1$ and the map $s\mapsto s\wedge t$ is a continuous surjection from $\mathcal{T}$ onto $[0,\omega_1]$ (we refer to \cite[Section 5]{RS C(K)} for further details). Consequently, Problem \ref{probl: Godefroy for C(K)} has a positive answer when $\K$ is a tree with the coarse-wedge topology. Incidentally, a positive answer was also given for adequate compacta in \cite[Section 5]{HRST}. \smallskip

Now, to the bad news. Problem \ref{probl: Godefroy for C(K)} would be solved for all Valdivia compacta if all scattered Valdivia compacta that contain $[0,\omega_1]$ admitted a continuous surjection onto it. However, \cite[Theorem 4.3]{RS C(K)} claims that, for a scattered Valdivia compactum $\K$, there is a continuous map from $\K$ onto $[0,\omega_1]$ if and only if $\K$ admits a P-point. As it turns out, the compactum $\K_\ro$ that we described in Section \ref{sec: WCG} fails to admit a P-point and therefore does not admit continuous surjections onto $[0,\omega_1]$ (we will give the definition of P-point and the justification of this claim in Section \ref{sec: semi-Eberlein}). This also explains why Problem \ref{probl: norming in Kro} is open and particularly relevant.

Incidentally, the scattered Valdivia compacta that contain $[0,\omega_1]$, but do not admit continuous surjections onto $[0,\omega_1]$ must necessarily contain weak P-points and fail to contain P-points, see \cite[Remark 4.7]{RS C(K)}. Therefore, they have to resemble the space $\K_\ro$, which once more hints at the importance of such a construction for Godefroy's problem. \smallskip

As we saw, the methods from \cite{AP, RS C(K)} alone are not efficient enough to yield a positive answer to Problem \ref{probl: Godefroy for C(K)} in its full generality. Hence, to conclude our discussion of the problem we shall highlight two problems whose positive answer would solve the main question in the positive, in combination with \cite{RS C(K)}. The first one concerns the assumption on $\K$ to be Valdivia that we made above.

\begin{problem}\label{probl: norming => Valdivia} Let $\K$ be a scattered compactum such that $\C(\K)$ admits a (countably) norming M-basis. Does it follow that $\K$ is Valdivia?
\end{problem}
Recall that by \cite[Theorem 5.3]{Kalenda survey}, this would be equivalent to the fact that $(B_{\C(\K)^*}, w^*)$ is Valdivia. For a general Banach space it is known that the existence of a countably norming M-basis does not imply that the dual ball is Valdivia; in fact, in the already mentioned \cite{Kalenda renorm Valdivia} Kalenda actually proved that the dual ball of $\C_0([0,\omega_1])$ is not Valdivia. However, here we are only asking about $\C(\K)$ spaces and the main concern is on norming M-bases.

The second problem is related to the containment of $[0,\omega_1]$ in $\K$.
\begin{problem}\label{probl: extension op} Let $\K$ be a scattered Valdivia compactum that contains $[0,\omega_1]$. Is there a linear extension operator from $\C([0,\omega_1])$ to $\C(\K)$?

More generally, does $\C([0,\omega_1])$ embed into $\C(\K)$?
\end{problem}

\begin{remark} A positive answer to both Problem \ref{probl: norming => Valdivia} and Problem \ref{probl: extension op} would answer in the positive Godefroy's problem for $\C(\K)$ spaces, Problem \ref{probl: Godefroy for C(K)}. Indeed, by the positive answer to the first problem, one can assume that $\K$ is Valdivia; then one has to distinguish two cases. If $\K$ does not contain $[0,\omega_1]$, it is Eberlein by Alster's \cite{Alster} and Deville--Godefroy's \cite{DG} results that we mentioned previously in the section. If $\K$ contains $[0,\omega_1]$, the positive answer to the second problem contradicts Theorem \ref{thm: AP-RS}.

By the same argument, the positive answer to Problem \ref{probl: extension op} alone implies a positive answer to Godefroy's problem under the additional assumption that $\K$ is Valdivia.
\end{remark}

To conclude this discussion, let us mention that Problem \ref{probl: extension op} has a negative answer if either of the two assumptions is removed. We are grateful to Antonio Avil\'es and Grzegorz Plebanek for pointing out the following examples to us.
\begin{example} Both examples depend on the fact that if $\K$ is c.c.c. (\emph{i.e.}, if every collection of mutually disjoint non-empty open subsets of $\K$ is countable), then $\C([0,\omega_1])$ does not embed in $\C(\K)$. Indeed, $\K$ being c.c.c. is equivalent to $\C(\K)$ not containing $c_0(\omega_1)$, \cite[Theorem 14.26]{FHHMZ}, while $\C([0,\omega_1])$ plainly does contain it (for instance, as $[0,\omega_1]$ is not c.c.c.).

For the first example it is then enough to take the cube $\{0,1\}^{\omega_1}$, which is a Valdivia compactum containing $[0,\omega_1]$. Moreover, a standard application of the Delta system lemma shows that $\{0,1\}^{\omega_1}$ is c.c.c.. Alternatively, $\C([0,\omega_1])$ does not embed in $\C(\{0,1\}^{\omega_1})$, because the latter has a norming M-basis, by \cite[Theorem 5.1]{HRST}.

The second example is less elementary and it requires the existence of a compactification $\K= \gamma\omega$ of $\omega$ whose remainder $\gamma\omega\setminus \omega$ is $[0,\omega_1]$. This is an example from \cite{FR compactification}, based on Magill's theorem \cite[Theorem 2.1]{Magill}; see \cite[pp. 141--145]{Walker} for details. Then it is clear that $\K$ is scattered and separable, hence c.c.c..
\end{example}

%-------------------------------------------------------%
%                                                       %
% 						SUBSPACES    					%
%                                                       %
%-------------------------------------------------------%
\section{The heredity problem}\label{sec: subspace}
In the previous section a central role was played by Alexandrov--Plichko's theorem (Theorem \ref{thm: AP}) that the Banach space $\C([0,\omega_1])$ does not admit a norming M-basis and its generalisation from \cite{RS C(K)} that it does not embed in a Banach space with norming M-basis. While the latter result is obviously formally stronger, it is apparently not known whether admitting a norming M-basis and embedding in a Banach space with norming M-bases are distinct conditions in general. In other words, to the best of the authors' knowledge, the following problem is still open.
\begin{problem}\label{probl: norming in subspace} Let $\X$ be a Banach space with norming M-basis and $\Z$ be a subspace of $\X$. Must $\Z$ admit a norming M-basis as well?
\end{problem}
In this section we will discuss some particular cases where a positive answer is available and some related questions. We also comment on a related question concerning countably norming M-bases, namely the subspace problem for Plichko spaces.

To begin with, it follows from Theorem \ref{thm: shrinking M-basis} that the existence of a shrinking M-basis passes to subspaces (because both the WLD and the Asplund classes are hereditary). On the other hand, a result of Plichko \cite[Theorem 2]{Plichko} (see, \emph{e.g.}, \cite[Theorem 4.63]{FHHMZ}) claims that the dual of every separable Banach space embeds in a space with an M-basis; in particular, $\ell_\infty$ is a subspace of a Banach space with M-basis. Therefore, the existence of M-bases does not pass to subspaces.

\begin{remark} In contrast to Plichko's result, $\ell_\infty$ does not embed in a Banach space with a norming M-basis. Indeed, as we saw in Section \ref{sec: WCG}, Banach spaces with norming M-basis admit an LUR norm, while $\ell_\infty$ does not have such a norm \cite[Theorem II.7.10]{DGZ}.
\end{remark}

By Theorem \ref{thm: separable M-basis}, the answer to Problem \ref{probl: norming in subspace} is clearly positive under the assumption that $\Z$ is separable. It turns out that this can be substantially generalised to the following result, that we mentioned already in previous sections.

\begin{theorem}[\cite{VWZ}]\label{thm: norming in WLD subspace} Let $\X$ be a Banach space with $\lambda$-norming M-basis and $\Z$ be a WLD subspace. Then $\Z$ admits a $\lambda$-norming M-basis.
\end{theorem}

As stated, the result appeared for the first time in \cite[Proposition 4.6]{VWZ} and \cite[Lemma 2.2]{Godefroy Rocky}. A weaker result, under the assumption that $\X$ is WCG, was proved in \cite[Proposition 6]{JZ Some notes WCG}; Vanderwerff \cite[Proposition 3.1]{Vanderwerff} proved the stronger claim that the $\lambda$-norming M-basis of $\Z$ can be extended to a $\lambda$-norming M-basis of $\X$. We sketch the proof, following \cite{Godefroy Rocky}.

\begin{proof} Let $\{e_\alpha; \p_\alpha\}_{\alpha\in \Gamma}$ be a $\lambda$-norming M-basis for $\X$. We first show that, without loss of generality, we can assume that $\lambda=1$. In fact, consider the equivalent norm on $\X$
\[ \nn{x}\coloneqq \sup \big\{|\langle\p,x\rangle|\colon \p\in {\rm span}\{\p_\alpha\}_{\alpha\in\Gamma},\, \|\p\|\leq 1 \big\}. \]
Then $\nn\cdot$ is an equivalent norm with $\lambda\n\leq \nn\cdot\leq \n$ and $\{e_\alpha; \p_\alpha\}_{\alpha\in \Gamma}$ is a $1$-norming M-basis in $(\X,\nn\cdot)$. Assuming the result for $\lambda=1$, it follows that $(\Z,\nn\cdot)$ admits a $1$-norming M-basis. It is easy to check that the same M-basis is $\lambda$-norming for $\Z$.

Thus, we assume that $\{e_\alpha; \p_\alpha\}_{\alpha\in \Gamma}$ is $1$-norming. Consider the functionals $g_\alpha\coloneqq \p_\alpha\cut_\Z$; clearly, ${\rm span}\{g_\alpha\}_{\alpha\in \Gamma}$ is $1$-norming for $\Z$. Moreover, by definition of M-basis, for every $z\in \Z$
\begin{equation}\label{eq: countable supp in Z}
    {\rm supp}(z)\coloneqq \{\alpha\in\Gamma\colon \langle g_\alpha, z\rangle\neq 0 \}\;\; \mbox{is at most countable.}
\end{equation}
Additionally, $\Z$ admits an M-basis $\{u_\beta; \psi_\beta\}_{\beta\in \Lambda}$; because $\Z$ is WLD, for all $\psi\in \Z^*$, the set
\begin{equation}\label{eq: countable supp in Z*}
    {\rm supp}(\psi)\coloneqq \{\beta\in\Lambda\colon \langle \psi, u_\beta\rangle\neq 0 \}\;\; \mbox{is also at most countable.}
\end{equation}
Conditions \eqref{eq: countable supp in Z} and \eqref{eq: countable supp in Z*} then imply the existence of an M-basis $\{x_\beta; f_\beta\}_{\beta\in \Lambda}$ in $\Z$ with the property that ${\rm span} \{x_\beta\}_{\beta\in \Lambda}= {\rm span} \{u_\beta\}_{\beta\in \Lambda}$ and ${\rm span} \{f_\beta\}_{\beta\in \Lambda}= {\rm span} \{g_\alpha\}_{\alpha\in \Gamma}$; in particular, ${\rm span} \{f_\beta\}_{\beta\in \Lambda}$ is $1$-norming, as desired. The proof of the existence of such an M-basis uses the well-known method to construct projectional resolutions of the identity by transfinite induction. However, its explanation would require too long of a digression, hence we omit the argument and we refer to \cite[Theorem 2.3]{VWZ} for the proof.
\end{proof}

\begin{remark} In particular, the WCG Banach space without norming M-basis from Theorem \ref{thm: H WCG not norming} does not even embed in a Banach space with norming M-basis. 
\end{remark}

Given Theorem \ref{thm: norming in WLD subspace}, the simplest candidate for a possible counterexample to Problem \ref{probl: norming in subspace} seems to be $\ell_1(\omega_1)$. Therefore, we isolate the following particular case.

\begin{problem} Does every subspace of $\ell_1(\Gamma)$, $\Gamma$ uncountable, admit a norming M-basis?
\end{problem}

In connection to this problem, let us recall that is not even known if subspaces of $\ell_1(\Gamma)$ are Plichko, which was asked by Kalenda in \cite[Question 4.44]{Kalenda survey}. Notice that hyperplanes, as well as closed sublattices, of $\ell_1(\Gamma)$ are $1$-Plichko, \cite[Example 4.39, Corollary 6.11]{Kalenda survey}.\smallskip

We now pass to the second part of the section where we briefly discuss the analogue to Problem \ref{probl: norming in subspace} for countably norming M-bases, or equivalently the heredity problem for Plichko spaces, that was asked for instance in \cite[Question 4.45]{Kalenda survey}.

\begin{problem}\label{probl: Plichko subspace} Let $\Z$ be a subspace of a Plichko Banach space $\X$. Must $\Z$ be Plichko?
\end{problem}

Differently from the problem for norming M-bases, this problem has been solved and a negative answer follows from Kubi\'s' results, \cite{Kubis retractions, Kubis subspace}. The solution is based upon two ingredients. The first one, of topological nature, is the result from \cite[Theorem 5.8]{Kubis retractions}, containing the construction of a $0$-dimensional linearly ordered Valdivia compact space $\mathcal{V}_{\omega_1}$ such that every linearly ordered Valdivia compactum is order preserving image of $\mathcal{V}_{\omega_1}$ (the notation $\mathcal{V}_{\omega_1}$ actually comes from \cite{Kubis subspace}); here it is perhaps worth noting that linearly ordered Valdivia compacta have weight at most $\omega_1$, \cite[Proposition 5.5]{Kubis retractions}. Then $\C(\mathcal{V}_{\omega_1})$ is $1$-Plichko, by \cite[Theorem 5.2]{Kalenda survey}. Now, consider the `connectification' $\K_{\omega_1}$ of $\mathcal{V}_{\omega_1}$ (see \cite[Section 5]{Kubis subspace} for the definition), which is a (two-to-one) order preserving continuous image of $\mathcal{V}_{\omega_1}$. Therefore, it follows that $\C(\K_{\omega_1})$ is a subspace of $\C(\mathcal{V}_{\omega_1})$. Finally, the last ingredient is the fact that $\C(\K_{\omega_1})$ is not Plichko, \cite[Theorem 5.1]{Kubis subspace}. Therefore, Plichko spaces are not closed under taking subspaces.

\begin{remark} At this point, it is of course natural to ask whether $\C(\mathcal{V}_{\omega_1})$ could admit a norming M-basis, since this would solve Problem \ref{probl: norming in subspace}. However, by the very property of $\mathcal{V}_{\omega_1}$, $[0,\omega_1]$ is a continuous image of $\mathcal{V}_{\omega_1}$, therefore $\C([0,\omega_1])$ is a subspace of $\C(\mathcal{V}_{\omega_1})$. Hence the latter does not admit a norming M-basis, due to Theorem \ref{thm: AP-RS}.

However, we can actually say more. In fact, another property of $\mathcal{V}_{\omega_1}$ is that it is `saturated' by copies of $[0,\omega_1]$, \cite[Theorem 5.8]{Kubis retractions}. Moreover, it is easy to see that every closed subset of a linearly ordered compactum admits a regular extension operator, \cite[Lemma 4.2]{Kubis subspace}. Therefore, $\C([0,\omega_1])$ is even $1$-complemented in $\C(\mathcal{V}_{\omega_1})$.
\end{remark}

On the other hand, some particular cases of Problem \ref{probl: Plichko subspace} are still open. To begin with, let us recall the following analogue of the classical problem whether being WCG is inherited from the bidual.
\begin{problem}[{\cite[Question 6.14]{Kalenda survey}}] Let $\X$ be a Banach space such that $\X^{**}$ is Plichko. Must $\X$ be Plichko? 
\end{problem}

The most important particular case of the heredity problem for Plichko spaces is the complemented subspace problem for Plichko spaces that we focus on next.
\begin{problem}\label{probl: complemented of Plichko} Let $\Z$ be a complemented subspace of a Plichko space $\X$. Must $\Z$ be Plichko?
\end{problem}

A positive answer to this is known under the additional assumption that $\X/ \Z$ is separable, \cite[Theorem 4.40]{Kalenda survey}. Moreover, Kubi\'s proved that every $1$-complemented subspace of a $1$-Plichko Banach space of density $\omega_1$ is $1$-Plichko as well, \cite[Theorem 6.3]{Kubis subspace}.

There is a connection between Problem \ref{probl: complemented of Plichko} and Theorem \ref{thm: AP} that we would like to emphasise, based on Kalenda's result that $\C([0,\omega_2])$ is not Plichko. More precisely, Kalenda \cite{Kalenda omega2} proved that the Banach space $\C([0,\eta])$ does not admit a countably norming M-basis for all ordinals $\kappa\leq \eta< \kappa\cdot \omega$, where $\kappa$ is a regular cardinal. However, it is not known whether the result holds for all $\eta\geq \omega_2$. A negative answer would be a strong counterexample to Problem \ref{probl: complemented of Plichko}, since $\C([0,\omega_2])$ is clearly a complemented subspace of $\C([0,\eta])$. Therefore, we reiterate the following question.
\begin{problem}[Kalenda, \cite{Kalenda omega2}] Is $\C([0,\eta])$ Plichko only for $\eta< \omega_2$?
\end{problem}

Notice that $\C([0,\eta])$ is not $1$-Plichko when $\eta\geq\omega_2$, because $[0,\eta]$ is not Valdivia for $\eta\geq \omega_2$ by \cite[Example 1.10]{Kalenda survey} and then the claim follows from \cite[Theorem 5.3]{Kalenda survey}; see also \cite[Theorem 5.10]{Kalenda PLMS}.

One might wonder whether, similarly as Section \ref{sec: C(K)}, one could improve Kalenda's result \cite{Kalenda omega2} by proving that $\C([0,\omega_2])$ does not embed in a Plichko space. At the moment, Kalenda's proof does not seem to give this stronger claim, therefore the following problem is also open.
\begin{problem}[\cite{Kalenda omega2, Kubis skeleton}] Does $\C([0,\omega_2])$ embed in a Plichko space?
\end{problem}

To conclude the section, let us return to Problem \ref{probl: norming in subspace} and observe that the analogue to Problem \ref{probl: complemented of Plichko} is apparently also open for norming M-bases.
\begin{problem} Let $\X$ be a Banach space with norming M-basis and $\Z$ be a complemented subspace of $\X$. Must $\Z$ admit a norming M-basis?    
\end{problem}

%-------------------------------------------------------%
%                                                       %
% 						SEMI-EBERLEIN 					%
%                                                       %
%-------------------------------------------------------%
\section{Semi-Eberlein compacta}\label{sec: semi-Eberlein}
As we saw in Section \ref{sec: non-sep classes}, the map $\ev$ from \eqref{eq: ev map} allows one to characterise several classes of non-separable Banach spaces in terms of topological properties of the dual unit ball with the $w^*$-topology. In this section we focus on the class of compact spaces whose definition is obtained by describing the image of $\ev$ when the M-basis is $1$-norming. Such compacta were introduced in \cite{KL}, where they are called semi-Eberlein compacta, and some results concerning them were recently obtained in \cite{AK_BLMS, CCS retraction, CRS, HRST}; still, their structure is rather poorly understood as of today.

\begin{definition}[\cite{KL}] A compact space is \emph{semi-Eberlein} if it is homeomorphic to a compact space $\K \subseteq [-1,1]^\Gamma$ such that $\K \cap c_0(\Gamma)$ is dense in $\K$.
\end{definition}

By the very definition, every Eberlein compactum is semi-Eberlein and every semi-Eberlein compactum is Valdivia. Moreover, every cube $[-1,1]^\Gamma$ is by definition semi-Eberlein and the ordinal interval $[0,\omega_1]$ is not semi-Eberlein, as we will discuss; thus these three classes are distinct. The table below illustrates the fact that semi-Eberlein compacta are obtained from Eberlein compacta via the same generalisation that leads from Corson to Valdivia compacta.

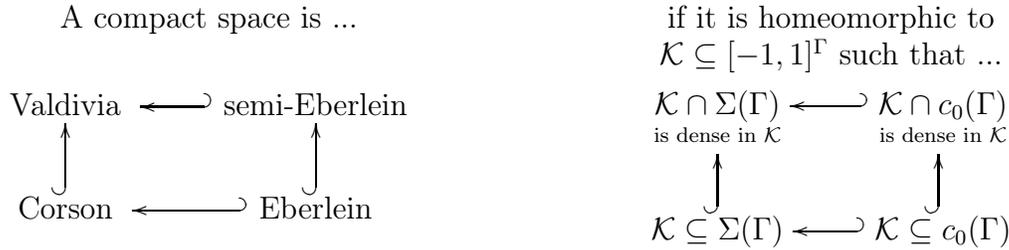
\begin{multicols}{2}\begin{center}
    A compact space is ...
    
    \columnbreak
    if it is homeomorphic to\\ $\K\subseteq [-1,1]^\Gamma$ such that ...
    \end{center}
\end{multicols}
\vspace{-3em}
\begin{multicols}{2}\begin{center}
    $${\xymatrix{ \text{ Valdivia } & \text{ semi-Eberlein } \ar@{_{(}->}[l]\\
    \text{ Corson } \ar@{^{(}->}[u] & \text{ Eberlein } \ar@{^{(}->}[u]\ar@{_{(}->}[l] }}$$
    \vspace{.3em}
    
    \columnbreak
    $${\xymatrix{ \underset{\text{is dense in }\K}{\K\cap \Sigma(\Gamma)}  & \; \underset{\text{is dense in }\K}{\K\cap c_0(\Gamma)} \ar@{_{(}->}[l] \\ 
    \K\subseteq \Sigma(\Gamma) \ar@{^{(}->}[u] & \;\K\subseteq c_0(\Gamma) \ar@{^{(}->}[u]\ar@{_{(}->}[l] }}$$
    \end{center}
\end{multicols}

Let us start once more by considering $[0,\omega_1]$. In order to explain why this compact space is not semi-Eberlein, we need to recall some topological notions. Recall that a point $p$ in a topological space $\mathcal{T}$ is a \emph{P-point} if $p$ is not isolated and for every countable family $(\mathcal{U}_k)_{k=1}^\infty$ of neighbourhoods of $p$, $\cap_{k=1}^\infty \mathcal{U}_k$ is a neighbourhood of $p$. A point $p\in \mathcal{T}$ is a \textit{weak P-point} if $p$ is not isolated and it is not a limit point of any countable set in $\mathcal{T} \setminus \{p\}$. An archetypal example of P-point is the point $\omega_1$ in $[0,\omega_1]$.

Kubi\'s and Leiderman proved in \cite[Theorem 4.2]{KL} that semi-Eberlein compacta do not admit P-points, whence it follows that $[0,\omega_1]$ fails to be semi-Eberlein. Incidentally, this result can be combined with a forcing argument to give an example of a Corson compactum that is not semi-Eberlein, \cite[Example 5.5]{KL}. It was then natural to ask whether semi-Eberlein compacta could admit weak P-points, \cite[Question 6.1]{KL}, which was answered in \cite[Section 4.3]{HRST}, once more by consideration of $\K_\ro$.

\begin{remark}\label{rmk: weak P point in Kro} By definition, the set $\F_\ro$ consisting of finite sets is dense in $\K_\ro$, thus the canonical embedding of $\K_\ro$ in $\{0,1\}^{\omega_1}$ witnesses that $\K_\ro$ is semi-Eberlein. In particular it admits no P-points, which is actually easy to check directly. However, the set $[0,\omega_1)$ is a weak P-point in $\K_\ro$, because every set in $\K_\ro\setminus\{[0,\omega_1)\}$ is countable. In this sense $\K_\ro$ behaves roughly like the ordinal interval $[0,\omega_1]$.
\end{remark}

Let us now turn to the consideration of semi-Eberlein dual balls. For instance, the dual ball of $\ell_1(\Gamma)$ is semi-Eberlein, as we already noted. Moreover, when $\X$ is WCG (or, more generally, subspace of WCG), $(B_{\X^*}, w^*)$ is semi-Eberlein as well. In particular, by taking $\X$ to be the Banach space from Theorem \ref{thm: H WCG not norming}, we see that $(B_{\X^*}, w^*)$ is semi-Eberlein and yet $\X$ does not admit any norming M-basis. Concerning the opposite implication, we have the following remark.

\begin{remark}\label{rmk: 1-norming gives semi-Eberlein} Assume that $\X$ admits a $1$-norming M-basis $\{e_\alpha; \p_\alpha\}_{\alpha\in \Gamma}$ and consider the embedding $\ev\colon (B_{\X^*},w^*)\to [-1,1]^\Gamma$. Clearly, 
\[ \ev[{\rm span}\{ \p_\alpha\}_{\alpha\in \Gamma}\cap B_{\X^*}]\subseteq c_{00}(\Gamma)\cap [-1,1]^\Gamma. \]
Therefore, \eqref{eq: norming in dual ball} shows that $(B_{\X^*},w^*)$ is semi-Eberlein. As we just observed, the converse fails to hold.
\end{remark}

To the best of our knowledge, this is all what is known concerning semi-Eberlein dual balls: known examples of semi-Eberlein dual balls either are actually Eberlein, or they come from $1$-norming M-bases. The main open problem here is clearly the one to characterise those Banach spaces whose dual ball is semi-Eberlein. Here we list a few more specific problems.
\begin{problem}\begin{enumerate}
    \item Let $\X$ be a Banach space with norming M-basis. Must $(B_{\X^*},w^*)$ be semi-Eberlein?
    \item Assume that $(B_{\X^*},w^*)$ is semi-Eberlein for every renorming of $\X$. Is $\X$ a subspace of a WCG Banach space?
    \item Are the dual balls of $\C([0,\omega_1])$ or $\C(\K_\ro)$ semi-Eberlein?
\end{enumerate}
\end{problem}

Recall that the dual ball of $\C_0([0,\omega_1])$ is not Valdivia, \cite{Kalenda renorm Valdivia}; thus the first problem has a negative answer for countably norming M-bases and Valdivia compacta. Concerning (2), note that if $\X$ is a subspace of a WCG space, its dual ball is Eberlein (and this holds under every renorming).

\begin{remark} Notice that in Remark \ref{rmk: 1-norming gives semi-Eberlein} we actually proved that $(B_{\X^*},w^*)$ embeds in $[-1,1]^\Gamma$ in such a way that $B_{\X^*}\cap c_{00}(\Gamma)$ is dense in $(B_{\X^*},w^*)$. Therefore it is natural to ask whether this behaviour of density of finitely supported vectors is common to all semi-Eberlein spaces; this problem was already asked in \cite[Problem 6.4]{KL}.
\end{remark}

\begin{problem} Let $\K$ be a semi-Eberlein compactum. Is there a homeomorphic embedding $\p\colon \K\to [-1,1]^\Gamma$ such that $\p(\K)\cap c_{00}(\Gamma)$ is dense in $\p(\K)$?

In particular, is the property true in the following cases?
\begin{enumerate}
    \item $\K$ is scattered;
    \item $\K$ is actually Eberlein, or even uniform Eberlein;
    \item in particular, $\K$ is the compactum constructed in \cite{Hajek}, see Theorem \ref{thm: K for no norming};
    \item $\K$ is the dual ball of a Banach space.
\end{enumerate}
\end{problem}

Notice that a positive answer to the problem has been recently given for $\K$ metrisable, \cite[Proposition 6.5]{MPZ}. Further, while our paper was in press, Avil\'es and Krupski solved (2) in the negative \cite[Example 4.10]{AK_BLMS}, thereby also answering \cite[Problem 6.4]{KL}. Concerning (1), recall that an Eberlein compactum is scattered if and only if it is \emph{strong Eberlein}, namely it is homeomorphic to a compact subset of $c_{00}(\Gamma)\cap \{0,1\}^\Gamma$, \cite{Alster}. This does not extend to semi-Eberlein compacta, since $\{0,1\}^\Gamma$ is not scattered, but $c_{00}(\Gamma)\cap \{0,1\}^\Gamma$ is clearly dense in $\{0,1\}^\Gamma$. \smallskip

To conclude, let us point out that, as a rule, techniques used to prove results for Valdivia or Eberlein compacta do not admit obvious modifications for semi-Eberlein compacta. Perhaps the unique exception is a Rosenthal-type characterisation of semi-Eberlein compacta, \cite[Proposition 3.1]{KL}, analogous to \cite[Theorem 3.1]{Rosenthal} (see also \cite[Proposition 1.9]{Kalenda survey}). For instance, it is not known if semi-Eberlein compacta admit a similar decomposition as the one obtained by Farmaki \cite{Farmaki} for Eberlein compacta, see \cite[Question 47]{CCS retraction}. Very little is also known concerning continuous images of semi-Eberlein compacta; the most interesting result here is perhaps that the class of Corson semi-Eberlein compacta is stable under continuous images, \cite[Theorem C]{CCS retraction} (notice that the said class is strictly larger than that of Eberlein compacta, \cite[Section 3]{KL}, \cite{Talagrand}). Some other explicit questions can be found in \cite[Section 6]{KL} (note that Question 1 and the second part of Question 6 there have been answered in \cite{HRST} and \cite{CCS retraction} respectively).

Here we limit ourselves to asking the following problem, which is an analogue of a well-known open question due to Avil\'es and Kalenda \cite[Problem 16]{AK problems} for Valdivia compacta.
\begin{problem} Can a scattered semi-Eberlein compactum contain a copy of $[0,\omega_2]$?
\end{problem}

\medskip
{\bf Acknowledgements.} The authors wish to express their gratitude to Antonio Avil\'es, Mari\'an Fabian, Gilles Godefroy, Chris Lambie-Hanson, Gilles Lancien, Grzegorz Plebanek, and Jacopo Somaglia for several helpful remarks on various parts of the paper.

%-------------------------------------------------------%
%                                                       %
% 						BIBLIOGRAPHY    				%
%                                                       %
%-------------------------------------------------------%


\begin{thebibliography}{FHHMZ}

\bibitem{AK} F.~Albiac and N.~Kalton, \emph{Topics in Banach space theory}, Graduate Texts in Mathematics, {\bf 233}. Springer, New York, 2006.

\bibitem{Alexandrov} A.G.~Alexandrov, \emph{Strong M-bases and equivalent norms in nonseparable Banach spaces}, Godishnik Vissh. Uchebn. Zaved. Prilozhna Mat. {\bf 19} (1983), 31--44.

\bibitem{AP} A.G.~Alexandrov and A.N.~Plichko,\emph{Connection between strong and norming Markushevich bases in nonseparable Banach spaces}, Mathematika {\bf 53} (2006), 321--328.

\bibitem{Alster} K.~Alster, \emph{Some remarks on Eberlein compacts}, Fund. Math. {\bf 104} (1979), 43--46.

\bibitem{AmirLind} D.~Amir and J.~Lindenstrauss, \emph{The structure of weakly compact sets in Banach spaces}, Ann. of Math. {\bf 88} (1968), 35--46.

\bibitem{Ansari} S.I.~Ansari, \emph{Existence of hypercyclic operators on topological vector spaces}, J. Funct. Anal. {\bf 148} (1997), 384--390.

\bibitem{ArMe WLD} S.A.~Argyros and S.K.~Mercourakis, \emph{On weakly Lindel\"{o}f Banach spaces}, Rocky Mountain J. Math. {\bf 23} (1993), 395--446.

\bibitem{AMN} S.A.~Argyros, S.K.~Mercourakis, and S.A.~Negrepontis, \emph{Functional-analytic properties of Corson-compact spaces}, Studia Math. {\bf 89} (1988), 197--229.

\bibitem{AK problems} A.~Avil\'es and O.F.K.~Kalenda, \emph{Compactness in Banach space theory---selected problems}, Rev. R. Acad. Cienc. Exactas F\'is. Nat. Ser. A Mat. RACSAM {\bf 104} (2010), 337--352.

\bibitem{AK_BLMS} A.~Avil\'es and M.~Krupski, \emph{On the class of NY compact spaces of finitely supported elements and related classes}, Bull. Lond. Math. Soc. {\bf 57} (2025), 1729--1748. 

\bibitem{BaKu} T.~Banakh and W.~Kubi\'s, \emph{Spaces of continuous functions over Dugundji compacta}, \href{https://arxiv.org/abs/math/0610795v2}{arXiv/0610795v2}.

\bibitem{Bekkali} M.~Bekkali, \emph{Topics in set theory. Lebesgue measurability, large cardinals, forcing axioms, rho-functions}, Lecture Notes in Mathematics, {\bf 1476}. Springer-Verlag, Berlin, 1991.

\bibitem{BRW} Y.~Benyamini, M.E.~Rudin, and M.~Wage, \emph{Continuous images of weakly compact subsets of Banach spaces}, Pacific J. Math. {\bf 70} (1977), 309--324.

\bibitem{BeSt} Y.~Benyamini and T.~Starbird, \emph{Embedding weakly compact sets into Hilbert space}, Israel J. Math. {\bf 23} (1976), 137--141.

\bibitem{Causey1} R.M.~Causey, \emph{The Szlenk index of injective tensor products and convex hulls}, J. Funct. Anal. {\bf 272} (2017), 3375--3409.

\bibitem{Causey2} R.M.~Causey, \emph{A note on the relationship between the Szlenk and  $w^*$-dentability indices of arbitrary  $w^*$-compact sets}, Positivity {\bf 21} (2017), 1507--1525.

\bibitem{Causey3} R.M.~Causey, \emph{The Szlenk index of $L_p(X)$ and $A_p$}, Proc. Amer. Math. Soc. {\bf 150} (2022), 4287--4302.

\bibitem{Causey4} R.M.~Causey, E.M.~Galego, and C.~Samuel, \emph{Szlenk index of $C(K)\hat{\otimes}_\pi C(L)$}, J. Funct. Anal. {\bf 282} (2022), 109414.

\bibitem{Causey5} R.M.~Causey and G.~Lancien, \emph{Prescribed Szlenk index of separable Banach spaces}, Studia Math. {\bf 248} (2019), 109--127.

\bibitem{CCS retraction} C.~Correa, M.~C\'uth, and J.~Somaglia, \emph{Characterization of (semi-)Eberlein compacta using retractional skeletons}, Studia Math. {\bf 263} (2022), 159--198.

\bibitem{CCS projection} C.~Correa, M.~C\'uth, and J.~Somaglia, \emph{Characterizations of weakly K-analytic and Va\v{s}\'ak spaces using projectional skeletons and separable PRI}, J. Math. Anal. Appl. {\bf 515} (2022), 126389.

\bibitem{CRS} C.~Correa, T.~Russo, and J.~Somaglia, \emph{Small semi-Eberlein compacta and inverse limits}, Topology Appl. {\bf 302} (2021), 107835.

\bibitem{CF dual Asplund} M.~C\'uth and M.~Fabian, \emph{Projections in duals to Asplund spaces made without Simons' lemma}, Proc. Amer. Math. Soc. {\bf 143} (2015), 301--308.

\bibitem{CF skeleton}  M.~C\'uth and M.~Fabian, \emph{Rich families and projectional skeletons in Asplund WCG spaces}, J. Math. Anal. Appl. {\bf 448} (2017), 1618--1632.

\bibitem{DHR fundamental} S.~Dantas, P.~H\'ajek, and T.~Russo, \emph{Smooth and polyhedral norms via fundamental biorthogonal systems}, Int. Math. Res. Not. IMRN {\bf 2023} (2023), 13909--13939.

\bibitem{DL total not norming} W.J.~Davis and J.~Lindenstrauss, \emph{On Total Nonnorming Subspaces}, Proc. Amer. Math. Soc. {\bf 31} (1972), 109--111.

\bibitem{DBS Rotund not LUR} C.A.~De Bernardi and J.~Somaglia, \emph{Rotund G\^{a}teaux smooth norms which are not locally uniformly rotund}, Proc. Amer. Math. Soc. {\bf 152} (2024), 1689--1701.

\bibitem{DG} R.~Deville and G.~Godefroy, \emph{Some applications of projective resolutions of identity}, Proc. Lond. Math. Soc. {\bf 67} (1993), 183--199.

\bibitem{DGZ} R.~Deville, G.~Godefroy, and V.~Zizler, \emph{Smoothness and Renormings in Banach Spaces}, Pitman Monographs and Surveys in Pure and Applied Mathematics {\bf 64}. Longman, Essex, 1993.

\bibitem{DB FPP} T.~Dom\'inguez Benavides, \emph{A renorming for nonseparable Banach spaces with the fixed point property}, J. Math. Anal. Appl. {\bf 350} (2009), 525--530.

\bibitem{Enflo} P.~Enflo, \emph{A counterexample to the approximation problem in Banach spaces}, Acta Math. {\bf 130} (1973), 309--317.

\bibitem{Fabian dual LUR} M.~Fabian, \emph{Each weakly countably determined Asplund space admits a Fr\'echet differentiable norm}, Bull. Austral. Math. Soc. {\bf 36} (1987), 367--374.

\bibitem{Fabian Gateaux PRI} M.~Fabian, \emph{Projectional resolution of the identity sometimes implies weakly compact generating}, Bull. Polish Acad. Sci. Math. {\bf 38} (1990), 117--120.

\bibitem{Fabian book} M.~Fabian, \emph{G\^{a}teaux differentiability of convex functions and topology. Weak Asplund spaces}, Canadian Mathematical Society Series of Monographs and Advanced Texts. A Wiley-Interscience Publication. John Wiley \& Sons, Inc., New York, 1997.

\bibitem{FG dual Asplund} M.~Fabian and G.~Godefroy, \emph{The dual of every Asplund space admits a projectional resolution of the identity}, Studia Math. {\bf 91} (1988), 141--151.

\bibitem{FGHZ} M.~Fabian, G.~Godefroy, P.~H\'ajek, and V.~Zizler, \emph{Hilbert-generated spaces}, J. Funct. Anal. {\bf 200} (2003), 301--323.

\bibitem{FGZ Asplund gen} M.~Fabian, G.~Godefroy, and V.~Zizler, \emph{A note on Asplund generated Banach spaces}, Bull. Polish Acad. Sci. Math. {\bf 47} (1999), 221--230.

\bibitem{FGZ} M.~Fabian, G.~Godefroy, and V.~Zizler, \emph{The structure of uniformly G\^{a}teaux smooth Banach spaces}, Israel J. Math. {\bf 124} (2001), 243--252.

\bibitem{FHHMZ} M.~Fabian, P.~Habala, P.~H\'{a}jek, V.~Montesinos, and V.~Zizler, \emph{Banach space theory. The basis for linear and nonlinear analysis}, CMS Books in Mathematics/Ouvrages de Math\'ematiques de la SMC. Springer, New York, 2011.

\bibitem{FM skeleton}  M.~Fabian and V.~Montesinos, \emph{WCG spaces and their subspaces grasped by projectional skeletons}, Funct. Approx. Comment. Math. {\bf 59} (2018), 231--250.

\bibitem{Farmaki} V.~Farmaki, \emph{The structure of Eberlein, uniformly Eberlein and Talagrand compact spaces in $\Sigma(\R^\Gamma)$}, Fund. Math. {\bf 128} (1987), 15--28.

\bibitem{FR compactification} S.P.~Franklin and M.~Rajagopalan, \emph{Some Examples in Topology}, Trans. Amer. Math. Soc. {\bf 155} (1971), 305--314.

\bibitem{Godefroy bdry} G.~Godefroy, \emph{Boundaries of a convex set and interpolation sets}, Math. Ann. {\bf 277} (1987),  173--184.

\bibitem{Godefroy Rocky} G.~Godefroy, \emph{Asplund spaces and decomposable nonseparable Banach spaces}, Rocky Mountain J. Math. {\bf 25} (1995), 1013--1024.

\bibitem{Grafakos} L.~Grafakos, \emph{Classical Fourier analysis}, Grad. Texts in Math., {\bf 249}. Springer, New York, 2014.

\bibitem{Gruenhage} G.~Gruenhage, \emph{Covering properties on $X^2\setminus \Delta$, $W$-sets, and compact subsets of $\Sigma$-products}, Topology Appl. {\bf 17} (1984), 287--304.

\bibitem{GMZ open} A.J.~Guirao, V.~Montesinos, and V.~Zizler, \emph{Open problems in the geometry and analysis of Banach spaces}, Springer, 2016.

\bibitem{GMZ renorm} A.J.~Guirao, V.~Montesinos, and V.~Zizler, \emph{Renormings in Banach spaces. A toolbox}, Monogr. Mat., {\bf 75}. Birkh\"{a}user/Springer, Cham, 2022.

\bibitem{Gulko} S.P.~Gul'ko, \emph{Properties of sets that lie in $\Sigma$-products}, Dokl. Akad. Nauk SSSR {\bf 237} (1977), 505--508.

\bibitem{Hajek} P.~H\'{a}jek, \emph{Hilbert generated Banach spaces need not have a norming Markushevich basis}, Adv. Math. {\bf 351} (2019), 702--717.

\bibitem{HMVZ} P.~H\'{a}jek, V.~Montesinos, J.~Vanderwerff, and V.~Zizler, \emph{Biorthogonal systems in Banach spaces}, CMS Books in Mathematics/Ouvrages de Math\'{e}matiques de la SMC, {\bf 26}. Springer, New York, 2008.

\bibitem{HR densely} P.~H\'ajek and T.~Russo, \emph{On densely isomorphic normed spaces}, J. Funct. Anal. {\bf 279} (2020), 108667.

\bibitem{HRST} P.~H\'{a}jek, T.~Russo, J.~Somaglia, and S.~Todor\v{c}evi\'c, \emph{An Asplund space with norming Marku\v{s}evi\v{c} basis that is not weakly compactly generated}, Adv. Math. {\bf 392} (2021), 108041.

\bibitem{HaTa} P.~H\'ajek and J.~Talponen, \emph{Smooth approximations of norms in separable Banach spaces}, Q. J. Math. {\bf 65} (2014), 957--969.

\bibitem{Haydon1} R.~ Haydon, \emph{A counterexample to several questions about scattered compact spaces}, Bull. Lond. Math. Soc. {\bf 22} (1990), 261--268.

\bibitem{Haydon2} R.~ Haydon, \emph{Baire trees, bad norms and the Namioka property}, Mathematika {\bf 42} (1995), 30--42.

\bibitem{Haydon3} R.~ Haydon, \emph{Trees in renorming theory}, Proc. Lond. Math. Soc. {\bf 78} (1999), 541--584.

\bibitem{Herzog} G.~Herzog, \emph{On linear operators having supercyclic vectors}, Studia Math. {\bf 103} (1992), 295--298.

\bibitem{JR selector} J.E.~Jayne and C.A.~Rogers, \emph{Borel selectors for upper semicontinuous set-valued maps}, Acta Math. {\bf 155} (1985), 41--79.

\bibitem{JZ X X*} K.~John and V.~Zizler, \emph{A renorming of dual spaces}, Israel J. Math. {\bf 12} (1972), 331--336.

\bibitem{JZ norming WCG} K.~John and V.~Zizler, \emph{Some remarks on nonseparable Banach spaces with Markushevich basis}, Comment. Math. Univ. Carolin. {\bf 15} (1974), 679--691.

\bibitem{JZ Smooth WCG} K.~John and V.~Zizler, \emph{Smoothness and its equivalents in weakly compactly generated Banach spaces}, J. Functional Analysis {\bf 15} (1974), 1--11.

\bibitem{JZ heredity} K.~John and V.~Zizler, \emph{On the heredity of weak compact generating}, Israel J. Math. {\bf 20} (1975), 228--236.

\bibitem{JZ WCG duality} K.~John and V.~Zizler, \emph{Weak compact generating in duality}, Studia Math. {\bf 55} (1976), 1--20.

\bibitem{JZ Some notes WCG} K.~John and V.~Zizler, \emph{Some notes on Marku\v{s}evi\v{c} bases in weakly compactly generated Banach spaces}, Compositio Math. {\bf 35} (1977), 113--123.

\bibitem{Johnson} W.B.~Johnson, \emph{No infinite dimensional P space admits a Markuschevich basis}, Proc. Amer. Math. Soc. {\bf 26} (1970), 467--468.

\bibitem{JL} W.B.~Johnson and J.~Lindenstrauss, \emph{Some remarks on weakly compactly generated Banach spaces}, Israel J. Math. {\bf 17} (1974), 219--230.

\bibitem{Kalenda renorm Valdivia} O.F.K.~Kalenda, \emph{An example concerning Valdivia compact spaces}, Serdica Math. J. {\bf 25} (1999), 131--140.

\bibitem{Kalenda image Valdivia} O.F.K.~Kalenda, \emph{Embedding the ordinal segment $[0,\omega_1]$ into continuous images of Valdivia compacta}, Comment. Math. Univ. Carolinae {\bf 40} (1999), 777--783.

\bibitem{Kalenda Valdivia equiv norm} O.F.K.~Kalenda, \emph{Valdivia compacta and equivalent norms}, Studia Math. {\bf 138} (2000), 179--191.

\bibitem{Kalenda survey} O.F.K.~Kalenda, \emph{Valdivia compact spaces in topology and Banach space theory}, Extracta Math. {\bf 15} (2000), 1--85.

\bibitem{Kalenda omega2} O.F.K.~Kalenda, \emph{M-bases in spaces of continuous functions on ordinals}, Colloq. Math. {\bf 92} (2002), 179--187.

\bibitem{Kalenda Valdivia non 1Plichko} O.F.K.~Kalenda, \emph{A new Banach space with Valdivia dual unit ball}, Israel J. Math. {\bf 131} (2002), 139--147.

\bibitem{Kalenda natural} O.F.K.~Kalenda, \emph{Natural examples of Valdivia compact spaces}, J. Math. Anal. Appl. {\bf 350} (2009), 464--484.

\bibitem{Kalenda PLMS} O.F.K.~Kalenda, \emph{Projectional skeletons and Markushevich bases}, Proc. Lond. Math. Soc. {\bf 120} (2020), 514--586.

\bibitem{Kubis retractions} W.~Kubi\'s, \emph{Compact spaces generated by retractions}, Topology Appl. {\bf 153} (2006), 3383--3396.

\bibitem{Kubis subspace} W.~Kubi\'s \emph{Linearly ordered compacta and Banach spaces with a projectional resolution of the identity}, Topology Appl. {\bf 154} (2007), 749--757.

\bibitem{Kubis skeleton} W.~Kubi\'s, \emph{Banach spaces with projectional skeletons}, J. Math. Anal. Appl. {\bf 350} (2009), 758--776.

\bibitem{KL} W.~Kubi\'s and A.~Leiderman, \emph{Semi-Eberlein compact spaces}, Topology Proc. {\bf 28} (2004), 603--616.

\bibitem{KuTr} D.~Kutzarova and S.~Troyanski, \emph{Reflexive Banach spaces without equivalent norms which are uniformly convex or uniformly differentiable in every direction}, Studia Math. {\bf 72} (1982), 91--95.

\bibitem{Lancien} G.~Lancien, \emph{A survey on the Szlenk index and some of its applications}, Rev. R. Acad. Cienc. Exactas F\'is. Nat. Ser. A Mat. RACSAM {\bf 100} (2006), 209--235.

\bibitem{Lind refl} J.~Lindenstrauss, \emph{On nonseparable reflexive Banach spaces}, Bull. Amer. Math. Soc. {\bf 72} (1966), 967--970.

\bibitem{Magill} K.D.~Magill, Jr., \emph{A note on compactifications}, Math. Z. {\bf 94} (1966), 322--325.

\bibitem{Mar} W.~Marciszewski, \emph{On sequential convergence in weakly compact subsets of Banach spaces}, Studia Math. {\bf 112} (1995), 189--194.

\bibitem{MPZ} W.~Marciszewski, G.~Plebanek, and K.~Zakrzewski, \emph{Digging into the classes of $\kappa$-Corson compact spaces},  Israel J. Math. (to appear), \href{https://arxiv.org/abs/2107.02513}{\texttt{arXiv:2107.02513}}.

\bibitem{Markushevich} A.I.~Markushevich, \emph{Sur les bases (au sens large) dans les espaces lin\'eaires}, C. R. (Doklady) Acad. Sci. URSS (N.S.) {\bf 41} (1943), 227--229.

\bibitem{MR Eberlein+Corson} E.~Michael and M.E.~Rudin, \emph{A note on Eberlein compacts}, Pacific J. Math. {\bf 72} (1977), 487--495.

\bibitem{MOTV} A.~Molt\'o, J.~Orihuela, S.L.~Troyanski, and M.~Valdivia, \emph{A nonlinear transfer technique for renorming}, Lecture Notes in Math. \textbf{1951}. Springer-Verlag, Berlin, 2009.

\bibitem{Moreno fund LUR} J.P.~Moreno, \emph{On the weak$^*$ Mazur intersection property and Fr\'echet differentiable norms on dense open sets}, Bull. Sci. Math. {\bf 122} (1998), 93--105.

\bibitem{Nyikos tree} P.~Nyikos, \emph{Various topologies on trees}, Proceedings of the Tennessee Topology Conference (Nashville, TN, 1996), 167--198, World Sci. Publ., River Edge, NJ, 1997. See also \href{https://arxiv.org/abs/math/0412554}{\texttt{arXiv/0412554}}.

\bibitem{Nyikos} P.~Nyikos, \emph{A Corson compact L-space from a Suslin tree}, Colloq. Math. {\bf 141} (2015), 149--156.

\bibitem{Orihuela} J.~Orihuela, \emph{On weakly Lindel\"{o}f Banach spaces}, Progress in Funct. Anal., K.D.~Bierstedt, J.~Bonet, J.~Horvath, and M.~Maestre, ed., Elsevier (1992).

\bibitem{Phelps book} R.R.~Phelps, \emph{Convex functions, monotone operators and differentiability}, Lecture Notes in Mathematics, {\bf 1364}. Springer-Verlag, Berlin, 1989.

\bibitem{Plichko} A.N.~Plichko, \emph{Projective resolutions, Markushevich bases and equivalent norms}, Matem. Zametki {\bf 34} (1983), 719--726.

\bibitem{Rosenthal} H.P.~Rosenthal, \emph{The heredity problem for weakly compactly generated Banach spaces}, Compositio Math. {\bf 28} (1974), 83--111.

\bibitem{RS tree} T.~Russo and J.~Somaglia, \emph{Weakly Corson compact trees}, Positivity {\bf 26} (2022), 33.

\bibitem{RS C(K)} T.~Russo and J.~Somaglia, \emph{Banach spaces of continuous functions without norming Markushevich bases}, Mathematika {\bf 69} (2023), 992--1010.

\bibitem{Semadeni} Z.~Semadeni, \emph{Banach spaces non-isomorphic to their Cartesian squares. II}, Bull. Acad. Polon. Sci. S\'er. Sci. Math. Astr. Phys. {\bf 8} (1960), 81--84.

\bibitem{Simons} S.~Simons, \emph{A convergence theorem with boundary}, Pacific J. Math. {\bf 40} (1972), 703--708.

\bibitem{Singer1} I.~Singer, \emph{Bases in Banach spaces. I}, Springer-Verlag, Berlin-New York, 1970.

\bibitem{JS tree1} J.~Somaglia, \emph{New examples of non-commutative Valdivia compact spaces}, Fund. Math. {\bf 243} (2018), 143--154.

\bibitem{JS tree2} J.~Somaglia, \emph{On compact trees with the coarse wedge topology}, Studia Math. {\bf 253} (2020), 283--306.

\bibitem{Talagrand} M.~Talagrand, \emph{Espaces de Banach faiblement $\mathcal{K}$-analytiques}, Ann. of Math. {\bf 110} (1979), 407--438.

\bibitem{Talagrand WCD} M.~Talagrand, \emph{A new countably determined Banach space}, Israel J. Math. {\bf 47} (1984), 75--80.

\bibitem{Terenzi} P.~Terenzi, \emph{Every separable Banach space has a bounded strong norming biorthogonal sequence which is also a Steinitz basis}, Studia Math. {\bf 111} (1994), 207–222.

\bibitem{T Acta} S.~Todor\v{c}evi\'{c}, \emph{Partitioning pairs of countable ordinals}, Acta Math. {\bf 159} (1987), 261--294.

\bibitem{T Walks} S.~Todor\v{c}evi\'{c}, \emph{Walks on Ordinals and Their Characteristics}, Progress in Mathematics {\bf 263}. Birkh\"{a}user Verlag Basel, 2007.

\bibitem{Troyanski LUR} S.L.~Troyanski, \emph{On locally uniformly convex and differentiable norms in certain non-separable Banach spaces}, Studia Math. {\bf 37} (1970/71), 173--180.

\bibitem{Valdivia} M.~Valdivia, \emph{Resolutions of the identity in certain Banach spaces}, Collect. Math. {\bf 39} (1988), 127--140.

\bibitem{Valdivia 1 norming} M.~Valdivia, \emph{On certain classes of Markushevich bases}, Arch. Math. {\bf 62} (1994), 445--458.

\bibitem{Vanderwerff} J.~Vanderwerff, \emph{Extensions of Marku\v{s}evi\v{c} bases}, Math. Z. {\bf 219} (1995), 21--30. 

\bibitem{VWZ} J.~Vanderwerff, J.H.M.~Whitfield, and V.~Zizler, \emph{Marku\u{s}evi\u{c} bases and Corson compacta in duality}, Canad. J. Math. {\bf 46} (1994), 200--211.

\bibitem{Vasak} L.~Va\v{s}\'ak, \emph{On one generalization of weakly compactly generated Banach spaces}, Studia Math. {\bf 70} (1981), 11--19.

\bibitem{Walker} R.C.~Walker, \emph{The Stone--\v{C}ech compactification}, Ergeb. Math. Grenzgeb., {\bf 83}. Springer-Verlag, New York-Berlin, (1974).

\bibitem{Zizler LUR} V.~Zizler, \emph{Locally uniformly rotund renorming and decompositions of Banach spaces}, Bull. Austral. Math. Soc. {\bf 29} (1984), 259--265.

\bibitem{Zizler} V.~Zizler, \emph{Nonseparable Banach spaces}, Handbook of the geometry of Banach spaces, Vol. 2, 1743--1816. North-Holland, Amsterdam, 2003. 

\end{thebibliography}
\end{document}